\documentclass[11pt]{amsart}
\usepackage{amsmath}
\usepackage{amsxtra}
\usepackage{amscd}
\usepackage{amsthm}
\usepackage{amsfonts}
\usepackage{amssymb}
\usepackage{eucal}

\input prepictex
\input pictex
\input postpictex

\newtheorem{thm}{Theorem}[section]
\newtheorem{lem}{Lemma}[section]
\newtheorem{cor}{Corollary}[section]
\newtheorem{prop}{Proposition}[section]

\theoremstyle{definition}

\theoremstyle{remark}
\newtheorem{rem}{Remark}[section]

\numberwithin{equation}{section}

\begin{document}

\newcommand{\thmref}[1]{Theorem~\ref{#1}}
\newcommand{\secref}[1]{Section~\ref{#1}}
\newcommand{\lemref}[1]{Lemma~\ref{#1}}
\newcommand{\propref}[1]{Proposition~\ref{#1}}
\newcommand{\corref}[1]{Corollary~\ref{#1}}
\newcommand{\remref}[1]{Remark~\ref{#1}}
\newcommand{\eqnref}[1]{(\ref{#1})}
\newcommand{\exref}[1]{Example~\ref{#1}}

\newcommand{\nc}{\newcommand}
\nc{\Z}{{\mathbb Z}} \nc{\C}{{\mathbb C}} \nc{\N}{{\mathbb N}}
\nc{\F}{{\mf F}} \nc{\Q}{\ol{Q}} \nc{\la}{\lambda}
\nc{\ep}{\epsilon} \nc{\h}{\mathfrak h} \nc{\n}{\mf n} \nc{\A}{{\mf
a}} \nc{\G}{{\mathfrak g}} \nc{\SG}{\overline{\mathfrak g}}
\nc{\D}{\mc D} \nc{\Li}{{\mc L}} \nc{\La}{\Lambda} \nc{\is}{{\mathbf
i}} \nc{\V}{\mf V} \nc{\bi}{\bibitem} \nc{\NS}{\mf N}
\nc{\dt}{\mathord{\hbox{${\frac{d}{d t}}$}}} \nc{\E}{\mc E}
\nc{\ba}{\tilde{\pa}} \nc{\half}{\frac{1}{2}} \nc{\mc}{\mathcal}
\nc{\mf}{\mathfrak} \nc{\hf}{\frac{1}{2}}
\nc{\hgl}{\widehat{\mathfrak{gl}}} \nc{\gl}{{\mathfrak{gl}}}
\nc{\hz}{\hf+\Z} \nc{\vac}{|0 \rangle}
\nc{\dinfty}{{\infty\vert\infty}} \nc{\SLa}{\overline{\Lambda}}
\nc{\SF}{\overline{\mathfrak F}} \nc{\SP}{\overline{\mathcal P}}
\nc{\U}{\mathfrak u} \nc{\SU}{\overline{\mathfrak u}}
\nc{\ov}{\overline}
\nc{\sL}{\ov{\mf{l}}}
\nc{\sP}{\ov{\mf{p}}}

\advance\headheight by 2pt

\title[Kostant's homology formula]{Howe duality and Kostant Homology Formula for infinite-dimensional Lie superalgebras}

\author[Shun-Jen Cheng]{Shun-Jen Cheng$^\dagger$}
\thanks{$^\dagger$Partially supported by an NSC-grant and an Academia Sinica Investigator grant}
\address{Institute of Mathematics, Academia Sinica, Taipei,
Taiwan 11529} \email{chengsj@math.sinica.edu.tw}

\author[Jae-Hoon Kwon]{Jae-Hoon Kwon$^{\dagger\dagger}$}
\thanks{$^{\dagger\dagger}$Partially supported by KRF-grant 2005-070-C00004.}
\address{Department of Mathematics, University of Seoul, 90, Cheonnong-Dong, Dongdaemun-gu, Seoul 130-743, Korea}
\email{jhkwon@uos.ac.kr}

\begin{abstract} \vspace{.3cm}  Using
Howe duality we compute explicitly Kostant-type homology groups for
a wide class of representations of the infinite-dimensional Lie
superalgebra $\hgl_{\infty|\infty}$ and its classical subalgebras at
positive integral levels.  We also obtain Kostant-type homology
formulas for the Lie algebra $\hgl_\infty$ at negative integral
levels. We further construct resolutions in terms of generalized
Verma modules for these representations.
\end{abstract}



\maketitle

\tableofcontents

\section{Introduction}

Howe duality \cite{H, H2} has found important applications in many
areas of mathematics where representation theory is an indispensable
tool. It is therefore not surprising that a concept as simple and
fundamental should prove to be useful and powerful in the
representation theory of Lie superalgebras. The first Howe dualities
involving Lie superalgebras already appeared in Howe's original
paper \cite{H}. It was followed by a systematic and in-depth study
of such dualities in \cite{CW1, CW2, LS, S2}. These works mainly
were concerned with the construction of Howe dualities involving Lie
superalgebras. Subsequent works focused on their applications; in
particular, we mention here the derivation of combinatorial
character formulas for Lie superalgebras in \cite{CL, CZ2}.

The purpose of the present paper is to demonstrate yet another
application of Howe duality to Lie superalgebras, namely the
computation of certain Kostant homology groups (and thus cohomology
groups) \cite{Ko}. As examples we compute them for a large class of
irreducible representations of $\hgl_{\infty|\infty}$
(cf.~\cite{KvL, KW2}) and its subalgebras of classical types
$\mf{b,c,d}$. Similar to character formulas, it is a rather
non-trivial task to compute such homology and cohomology groups for
Lie (super)algebras directly. Indeed such formulas and the
Euler-Poincar\'e principle immediately imply character formulas as
well.

Let us provide more details below.  In the remainder of this section
it will be convenient to restrict most of our discussion  to the
case of $\hgl_{\infty|\infty}$ for simplicity's sake. Consider the
Lie algebra $\gl_\infty$ consisting of matrices $(a_{ij})$ with
$i,j\in\Z$, and $a_{ij}=0$ for all but finitely many $a_{ij}$'s. Let
$\G=\hgl_\infty$ denote its well-known central extension by a
one-dimensional center (see \eqnref{cocycle}). Topological
completions of $\G$, and its subalgebras of classical types
$\mf{b,c,d}$, have found remarkable applications in mathematical
physics since their introduction by the Kyoto school \cite{DJKM},
and the results in this paper remain valid for these completions as
well, since all irreducible representations considered in this paper
are also irreducible representations of the corresponding
completions. Introduce the subalgebras $\gl_{\le 0}$, $\gl_{>0}$ and
$\U_-$ consisting of matrices $(a_{ij})$ with $i,j\le 0$, $i,j>0$
and $i> 0, j\le 0$, respectively. Let $\F^{\ell}$ denote the Fock
space generated by $\ell$ pairs of free fermions.  Then it is known
that there is a joint action of $\G\times {\rm GL}(\ell)$ on
$\F^\ell$. Indeed $\left(\G, {\rm GL}(\ell)\right)$ forms a
reductive dual pair \cite{F}, giving a one-to-one correspondence
between certain irreducible highest weight representations of $\G$
and the irreducible rational representations of ${\rm GL}(\ell)$.
Now the irreducible rational representations of ${\rm GL}(\ell)$ are
parameterized by $\mc{P}({\rm GL}(\ell))$, the set of generalized
partitions of length $\ell$. We denote the $\G$-module corresponding
to $\la\in\mc{P}({\rm GL}(\ell))$ by $L(\G,\Lambda^{\mf{a}}(\la))$.
Since the derived subalgebra is a generalized Kac-Moody algebra, the
homology theory of integrable modules over generalized Kac-Moody
algebra (cf.~\cite{J}) is applicable to $\G$, giving us a precise
description of the $\gl_{\le 0}\oplus\gl_{>0}$-module structure of
the $\mf{u}_-$-homology groups ${\rm
H}_k\left(\mf{u}_-;L\left(\G,\Lambda^{\mf{a}}(\la)\right)\right)$
for $k\in\mathbb{N}$.

Now consider the super analogue $\gl_{\infty|\infty}$ consisting of
matrices $(a_{rs})$ with $r,s\in\hf\Z$, and $a_{rs}=0$ for all but
finitely many $a_{rs}$'s. We regard $(a_{rs})$ as the matrix of a
linear transformation on the space $\C^{\infty|\infty}$ with respect
to a basis $\{\,e_r\,\vert \,r\in\hf\Z\,\}$, where the
$\Z_2$-gradation is given by  ${\rm deg}e_r=2r$ modulo $2$. There is
a similar central extension $\SG=\hgl_{\infty|\infty}$ (see
\eqnref{cocycle-super}), and we can define subalgebras
$\ov{\gl}_{\le 0}$, $\ov{\gl}_{>0}$, and $\ov{\U}_-$ in an analogous
fashion. Let $\ov{\F}^{\ell}$ denote the Fock space generated by
$\ell$ pairs of free fermions and $\ell$ pairs of free bosons.  Then
there exists a joint action of $\SG\times {\rm GL}(\ell)$ on
$\ov{\F}^\ell$, which also forms a Howe dual pair \cite{CW2}. This
provides another one-to-one correspondence between certain
irreducible highest weight representations of $\SG$ and the
irreducible rational representations of ${\rm GL}(\ell)$. We denote
the $\SG$-module corresponding to $\la\in\mc{P}({\rm GL}(\ell))$ by
$L(\SG,\SLa^{\mf{a}}(\la))$. Thus these two Howe dualities establish
a one-to-one correspondence
\begin{equation*}
L(\G,\La^{\mf{a}}(\la))\stackrel{1{-}1}{\longleftrightarrow}
L(\SG,\SLa^{\mf{a}}(\la)),\quad\la\in\mc{P}({\rm GL}(\ell)).
\end{equation*}
This correspondence enables us to determine the $\ov{\gl}_{\le
0}\oplus\ov{\gl}_{>0}$-module structure of ${\rm
H}_k(\ov{\U}_-;L(\SG,\SLa^{\mf{a}}(\la)))$ in terms ${\rm
H}_k\left({\U}_-;L\left(\G,\La^{\mf{a}}(\la)\right)\right)$, for all
$k\in\N$, in \thmref{mainthm}. We note that in contrast to Lie
algebras the derived subalgebra of $\hgl_{\infty|\infty}$ is not a
generalized Kac-Moody superalgebra, and thus a direct generalization
of the classical homology theory is not possible.

To our best knowledge such a similar relationship on the level of
homology groups between Lie algebras and Lie superalgebras was first
observed for the irreducible tensor representations of the general
linear superalgebra in \cite{CZ1}.  It is our hope that the results
presented here may lead to a better understanding of the
relationship between the representation theories of Lie algebras and
Lie superalgebras (e.g.~in the spirit of \cite{CWZ}), which is our
motivation for this investigation.

Besides Howe duality involving an infinite-rank affine Kac-Moody
algebras, worked out systematically in \cite{W1}, Aribaud's method
for computing homology from the character formula \cite{A, L} also
plays an important role in our approach. Together with elementary
symmetric function theory they provide us with all the necessary
tools to reduce the problem to a comparison of eigenvalues of the
classical and the super Casimir operators on the chain spaces.

We conclude with an outline of the paper. In Sections \ref{IDLA} and
\ref{IDLSA} the Howe dualities involving the infinite-dimensional
Lie (super)algebras are recalled. In \secref{tensor:repn} we discuss
in detail the category of tensor representations of $\gl(m|n)$, for
$n\in\N\cup\{\infty\}$. Although the results therein may be known to
experts, we have decided to include them, as we were unable to find
references for them. These results are then used in subsequent
sections. In Section \ref{casimir:op} invariant symmetric
non-degenerate bilinear forms are introduced for the Lie algebras
and superalgebras under discussion, based on which the Casimir
operators are defined. We remark that our Casimir operators  differ
slightly from the usual ones, as this makes subsequent comparisons
of their eigenvalues simpler. The homology groups for
$\hgl_{\infty|\infty}$ and its classical subalgebras are then
computed in \secref{sec:homology}. In \secref{sec:resolution}
resolutions in terms of generalized Verma modules are constructed.
In \secref{sec:negative} the $\hgl_\infty$-modules appearing in the
bosonic Fock space are discussed, and their Kostant-type homology
groups are computed. In \secref{final:notation} we collect some
notation for the convenience of the reader.

Finally, all vector spaces, algebras et cetera are over the complex
numbers $\C$.
\medskip

{\bf Acknowledgements.} The first author thanks NCTS for partial
support. The second author thanks the Institute of Mathematics,
Academia Sinica, for hospitality and support, where part of this
work was completed. Both authors thank the referees for helpful
comments.

\section{Infinite-dimensional Lie algebras}\label{IDLA}

\subsection{Notation and conventions}

Denote by $\mathcal P^+$ the set of partitions. Given $\ell\in\N$, a
non-increasing sequence $\la=(\la_1,\la_2,\cdots,\la_\ell)$ of
$\ell$ integers with
\begin{equation}\label{gpartition}
\la_1\ge\la_2\ge\cdots\ge\la_i> 0=\cdots=0>\la_j\ge\cdots\ge\la_\ell
\end{equation}
will be called a {\em generalized partition of length} $\ell$.
Similar to ordinary partitions, a generalized partition $\la$ may be
viewed as a left-adjusted diagram with $\ell$ rows such that the
$k$-th row contains $\la_k$ cells. Here we use the convention that
$\la_k\ge 0$ (resp. $\la_k\le 0$) means we have $|\la_k|$ boxes to
the right (resp. left) of a fixed vertical line indicating $0$
cells. As an example consider $\la=(5,3,2,1,-1,-2)$ with $\ell=6$.
The corresponding diagram is given in \eqnref{standard} below.
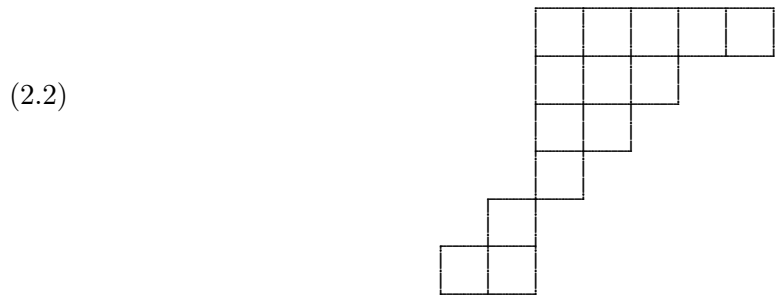
\begin{equation}\label{standard}
{\beginpicture \setcoordinatesystem units <1.5pc,1.5pc> point at 0 2
\setplotarea x from -1.5 to 1.5, y from -2 to 4 \plot 0 0 0 4 1 4 1
0 0 0 / \plot 1 1 2 1 2 4 1 4 / \plot 2 2 3 2 3 4 2 4 / \plot 3 3 5
3 5 4 3 4 / \plot 0 1 1 1 / \plot 0 2 2 2 / \plot 0 3 3 3 / \plot 4
3 4 4 / \plot 0 0 -1 0 / \plot 0 0 0 -2 / \plot -1 0 -1 -2 / \plot 0
-2 -2 -2 / \plot 0 -1 -2 -1 / \plot -2 -2 -2 -1 /
\endpicture}
\end{equation}
For a generalized partition $\la$ of length $\ell$, let $\la_{k}'$
be the length of the $k$-th column of $\la$. We use the convention
that the first column of $\la$ of the form \eqnref{gpartition} is
the first column of the partition
$\la_1\ge\la_2\ge\cdots\ge\la_{j-1}$. The column to its right is the
second column of $\la$, while the column to its left is the zeroth
column and the column to the left of its zeroth column is the
$-1$-st column et cetera.  We also use the convention that a
non-positive column has non-positive length. For our example
\eqnref{standard}, we have $\la_{-1}'=-1$, $\la_0'=-2$, $\la_1'=4$
et cetera. The {\em size} of a generalized partition $\la$ is
$\sum_j\la_j$, and is denoted by $|\la|$. For a partition
$\la=(\la_1,\la_2,\ldots)$, $\la'=(\la'_1,\la_2',\ldots)$ and
$l(\la)$ denote the conjugate and the length of $\la$, respectively.

\subsection{Preliminaries on $\widehat{\gl}_\infty$ and its subalgebras of classical
types}\label{class:subalgebras}

Below we recall the infinite-dimensional Lie algebra
$\widehat{\gl}_\infty$ and its subalgebras $\mathfrak{b}_{\infty}$,
$\mathfrak{c}_{\infty}$, $\mathfrak{d}_{\infty}$. We will say that
these algebras are of types $\mf{a,b,c,d}$, respectively. The
following notation will be assumed throughout the paper.
\begin{itemize}
\item[$\cdot$] $\mathfrak{g}$ : the Lie algebra
$\widehat{\gl}_\infty$, or its subalgebra $\mathfrak{b}_{\infty}$,
$\mathfrak{c}_{\infty}$, $\mathfrak{d}_{\infty}$,

\item[$\cdot$] $\mathfrak{h}$ : a Cartan subalgebra of $\mathfrak{g}$,

\item[$\cdot$] $I$ : the index set for simple roots,

\item[$\cdot$] $\Pi=\{\,\alpha_i\,|\,i\in I\,\}$ : the set of
simple roots,

\item[$\cdot$] $\Pi^{\vee}=\{\,\alpha_i^{\vee}\,|\,i\in I\,\}$ : the set of
simple coroots,

\item[$\cdot$] $\Delta^{+}$ : the set of positive roots.



\end{itemize}

\subsubsection{The Lie algebra
$\widehat{\gl}_\infty$}\label{rhoc:aux1} Let $\mathbb{C}^{\infty}$
be the vector space over $\C$ with a basis
$\{\,e_i\,|\,i\in\mathbb{Z}\,\}$ so that an element in ${\rm
End}(\C^\infty)$ may be identified with a matrix $(a_{ij})$
($i,j\in\Z$). Let $E_{ij}$ be the matrix with $1$ at the $i$-th row
and $j$-th column and zero elsewhere. Let $\gl_\infty$ denote the
subalgebra consisting of $(a_{ij})$ with finitely many non-zero
$a_{ij}$'s. Then $\gl_\infty$ in matrix form is spanned by $E_{ij}$
($i,j\in\Z$). Denote by $\widehat{\gl}_\infty=\gl_\infty\oplus\C K$
the central extension of $\gl_\infty$ by a one-dimensional center
$\C K$ given by the $2$-cocycle (cf.~\cite{DJKM, K})
\begin{equation}\label{cocycle}
\alpha(A,B):={\rm Tr}([J,A]B),
\end{equation}
where $J=\sum_{i\le 0}E_{ii}$. The derived subalgebra of
$\widehat{\gl}_\infty$ is an infinite-rank Kac-Moody algebra
\cite{K}. Note that  $\h=\sum_{i\in \Z}\C E_{ii}\oplus\C K$. Put
$I=\mathbb{Z}$. Then we have
\begin{align*}
\Pi^{\vee}=\{\,& \alpha_i^{\vee}=E_{ii}-E_{i+1,i+1}+\delta_{i0}K\ (i\in I) \, \}, \\
\Pi=\{\,& \alpha_i=\epsilon_i-\epsilon_{i+1} \ (i\in I)\, \}, \\
\Delta^+ =\{\,& \epsilon_i-\epsilon_j\ (i,j\in I, i<j) \,\},
\end{align*}
where $\epsilon_i\in\h^*$ is determined by $\langle
\epsilon_i,E_{jj}\rangle=\delta_{ij}$ and
$\langle\epsilon_i,K\rangle=0$.  Furthermore, for $i\in I$, let
$\Lambda^{\mf a}_i$ be the $i$-th fundamental weight of
$\widehat{\gl}_\infty$. That is, we have
\begin{align*}
\Lambda^{\mf a}_i=
\begin{cases}
\Lambda^{\mf a}_0-\sum_{k=i+1}^0\epsilon_k, & \text{if  $i< 0$}, \\
\Lambda^{\mf a}_0+\sum_{k=1}^i\epsilon_k, & \text{if  $i> 0$},
\end{cases}
\end{align*}
where $\Lambda^{\mf a}_0\in\h^*$ is determined by
$\langle\Lambda^{\mf a}_0,K\rangle=1$ and $\langle\Lambda^{\mf
a}_0,E_{jj}\rangle=0$, for all $j\in \Z$. Let $\rho_{c}\in\h^*$ be
determined by $\langle \rho_c,E_{jj}\rangle=-j$, for all $j\in\Z$,
and $\langle \rho_c,K\rangle=0$, so that we have $\langle
\rho_c,{\alpha}^{\vee}_i\rangle=1$, for all $i\in I$.

By assigning degree $0$ to the Cartan subalgebra and setting ${\rm
deg}E_{ij}=j-i$, we equip $\hgl_\infty$ with a $\Z$-gradation:
$\hgl_\infty=\bigoplus_{k\in\Z}(\hgl_\infty)_k$. This leads to the
following triangular decomposition:
\begin{equation*}
\hgl_\infty=(\hgl_\infty)_+\oplus (\hgl_\infty)_0\oplus
(\hgl_\infty)_-,
\end{equation*}
where $(\hgl_\infty)_{\pm}=\bigoplus_{k\in\pm\N}(\hgl_\infty)_k$ and
$(\hgl_\infty)_0=\mf{h}$.

\subsubsection{The Lie algebras $\mathfrak{b}_{\infty}$, $\mathfrak{c}_{\infty}$,
$\mathfrak{d}_{\infty}$}\label{rhoc:aux2} For convenience of the
reader we will briefly review the classical subalgebras of
$\widehat{\gl}_\infty$ (cf.~\cite{K}). For $\mf{x}\in\{\mathfrak{
b,c,d}\}$, let $\overline{\mf{x}}_{\infty}$ be the subalgebra of
${\gl}_\infty$ preserving the following bilinear form on
$\mathbb{C}^{\infty}$:
\begin{equation*}
 (e_i|e_j)=
\begin{cases}
 \delta_{i, -j}, & \text{if $\mf{x}=\mathfrak{b}$}, \\
(-1)^i\delta_{i,1-j}, & \text{if $\mf{x}=\mathfrak{c}$}, \\
 \delta_{i,1-j}, & \text{if $\mf{x}=\mathfrak{d}$},
\end{cases}\quad i,j\in\Z.
\end{equation*}
Let $\mf{x}_{\infty}=\overline{\mf{x}}_{\infty}\oplus \mathbb{C}K$
be the central extension of $\overline{\mf{x}}_{\infty}$ determined
by the restriction of the two-cocycle \eqnref{cocycle}. Then
$\mf{x}_{\infty}$ has a natural triangular decomposition induced
from $\widehat{\mathfrak{gl}}_{\infty}$:
\begin{equation*}
\mf{x}_{\infty}=(\mf{x}_{\infty})_+\oplus (\mf{x}_{\infty})_0 \oplus
(\mf{x}_{\infty})_-,
\end{equation*}
where $(\mf{x}_{\infty})_{\pm}=\mf{x}_{\infty}\cap
(\widehat{\mathfrak{gl}}_{\infty})_{\pm}$ and
$(\mf{x}_{\infty})_0=\mf{x}_{\infty} \cap
(\widehat{\mathfrak{gl}}_{\infty})_0$. For $i\in\mathbb{N}$, let
\begin{equation*}
\widetilde{E}_i=
\begin{cases}
E_{ii}-E_{-i,-i}, & \text{if $\mf{x}=\mathfrak{b}$}, \\
E_{ii}-E_{1-i,1-i}, & \text{if $\mf{x}=\mathfrak{c,d}$}.
\end{cases}
\end{equation*}
Note that
$\mathfrak{h}=(\mf{x}_{\infty})_0=\sum_{i\in\N}\C\widetilde{E}_i\oplus
\C K$. We may regard $\epsilon_i\in
(\widehat{\mathfrak{gl}}_{\infty})_0^*$ as an element in
$(\mf{x}_{\infty})_0^*$ via restriction so that
$\langle\epsilon_i,\widetilde{E}_j\rangle=\delta_{ij}$ for $i,j\in
\mathbb{N}$. Put $I=\mathbb{Z}_{+}$. Then we have \vskip 3mm

$\bullet$ $\mathfrak{b}_{\infty}$
\begin{align*}
\Pi^{\vee}=\{\,& \alpha_0^{\vee}=-2\widetilde{E}_1+2K, \
  \alpha_i^{\vee}=\widetilde{E}_i-\widetilde{E}_{i+1} \ (i\in\mathbb{N})\, \}, \\
\Pi=\{\,& \alpha_0=-\epsilon_1, \
 \alpha_i=\epsilon_i-\epsilon_{i+1} \ (i\in\mathbb{N})\, \}, \\
\Delta^+ =\{\,& \pm\epsilon_i-\epsilon_j\ ,\ -\epsilon_i\ (i,j\in\N,
i<j) \,\}.
\end{align*}

$\bullet$ $\mathfrak{c}_{\infty}$
\begin{align*}
\Pi^{\vee}=\{\,& \alpha_0^{\vee}=-\widetilde{E}_1+ K,  \
  \alpha_i^{\vee}=\widetilde{E}_i-\widetilde{E}_{i+1}  \ (i\in\mathbb{N})\, \}, \\
\Pi=\{\,& \alpha_0=-2\epsilon_1 , \
 \alpha_i=\epsilon_i-\epsilon_{i+1} \ (i\in\mathbb{N})\, \}, \\
\Delta^+ =\{\,& \pm\epsilon_i-\epsilon_j\ ,\ -2\epsilon_i\
(i,j\in\N, i<j) \,\}.
\end{align*}

$\bullet$ $\mathfrak{d}_{\infty}$
\begin{align*}
\Pi^{\vee}=\{\,& \alpha_0^{\vee}=-\widetilde{E}_1-\widetilde{E}_2+
2K, \   \alpha_i^{\vee}=\widetilde{E}_i-\widetilde{E}_{i+1}  \ (i\in\mathbb{N})\, \}, \\
\Pi=\{\,& \alpha_0=-\epsilon_1-\epsilon_2, \
 \alpha_i=\epsilon_i-\epsilon_{i+1} \ (i\in\mathbb{N})\, \}, \\
\Delta^+ =\{\,& \pm\epsilon_i-\epsilon_j\   (i,j\in\N, i<j) \,\}.
\end{align*}\vskip 3mm

For $i\in I$, we denote by $\La^\mf{x}_i$ the $i$-th fundamental
weight for $\mf{x}_{\infty}$, that is,
$\langle\La^\mf{x}_i,\alpha^{\vee}_j\rangle=\delta_{ij}$ ($j\in I$),
where $\La^\mf{x}_0\in \mf{h}^*$ is determined by
$\langle\La^\mf{x}_0,\widetilde{E}_i\rangle=0$ for $i\in\N$ and
$\langle\La^\mf{x}_0,K\rangle=r$ with $r=\hf, 1, \hf$, for
$\mf{x}=\mathfrak{b},\mathfrak{c},\mathfrak{d}$, respectively. In
fact, we have

\begin{align*}
\La^{\mf b}_i&=2\La^{\mf b}_0+\epsilon_1+\cdots+\epsilon_i, \ \ \
i\geq 1, \\
\La^{\mf c}_i&=\La^{\mf c}_0+\epsilon_1+\cdots+\epsilon_i, \ \ \ \ \
i\geq 1, \\
\La^{\mf d}_i &=
\begin{cases}
\La^{\mf d}_0+\epsilon_1, & \text{if $i=1$}, \\
2\La^{\mf d}_0+\epsilon_1+\cdots+\epsilon_i, & \text{if $i>1$}.
\end{cases}
\end{align*}
We let $\rho_c\in\h^*$ be determined by
\begin{equation*}
\begin{aligned}
&\langle \rho_c,\widetilde{E}_{j}\rangle=
\begin{cases}
-j+\hf, & \text{for $\mf{b}_{\infty}$}, \\
 -j, & \text{for $\mf{c}_{\infty},\mf{d}_{\infty}$},
\end{cases} j\in\N, \quad \
\langle \rho_c,K\rangle =
\begin{cases}
\ \, 0, & \text{for $\mf{b}_{\infty},\mf{c}_{\infty}$}, \\
-1, & \text{for $\mf{d}_{\infty}$}.
\end{cases}
\end{aligned}
\end{equation*}

\vskip 3mm

\subsection{Classical dual pairs on infinite-dimensional Fock spaces}\label{classical:dualpairs}
Let $\mf{g}$ be one of the Lie algebras given in
\secref{class:subalgebras}. Let $\La\in\h^*$ be given. By standard
arguments there is a unique irreducible highest weight
representation of $\G$ of highest weight $\La$, which will be
denoted by $L(\G,\La)$.

We fix a positive integer $\ell\geq 1$ and consider $\ell$ pairs of
free fermions $\psi^{\pm,i}(z)$ with $i=1,\cdots,\ell$. That is, we
have
\begin{align*}
\psi^{+,i}(z)&=\sum_{n\in\Z}\psi^{+,i}_nz^{-n-1},\quad\quad\
\psi^{-,i}(z)=\sum_{n\in\Z}\psi^{-,i}_nz^{-n},
\end{align*}
with non-trivial commutation relations
$[\psi^{+,i}_m,\psi^{-,j}_n]=\delta_{ij}\delta_{m+n,0}$. Let
$\F^\ell$ denote the corresponding Fock space generated by the
vaccum vector $|0\rangle$, which is annihilated by
$\psi^{+,i}_n,\psi^{-,i}_m$ for $n\ge 0$ and $m>0$.

We introduce a neutral fermionic field
$\phi(z)=\sum_{n\in\Z}\phi_nz^{-n-1}$ with non-trivial commutation
relations $[\phi_m,\phi_n]=\delta_{m+n,0}$. Denote by $\F^{\hf}$ the
Fock space of $\phi(z)$ generated by a vacuum vector that is
annihilated by $\phi_m$ for $m\geq 0$. We denote by $\F^{\ell+\hf}$
the tensor product of $\F^{\ell}$ and $\F^{\hf}$.

We assume that $x_n$ ($n\in\Z$) and $z_i$ ($i=1,\ldots,\ell$) are
formal indeterminates.

\subsubsection{The $(\widehat{\gl}_\infty,{\rm GL}(\ell))$-duality}

Let ${\rm GL}(\ell)$ be the general linear group of rank $\ell$. The
space of complex $\ell\times\ell$ matrices forms the Lie algebra
$\mf{gl}(\ell)$ of ${\rm GL}(\ell)$. We denote by $e_{ij}$ ($1\leq
i,j\leq \ell$) the elementary matrix with 1 in the $i$-th row and
$j$-th column and $0$ elsewhere. Then $H=\sum_{i}\C e_{ii}$ is a
Cartan subalgebra, while $\sum_{i\le j}\C e_{ij}$ is a Borel
subalgebra containing $H$. Recall that an irreducible rational
representations of $\mf{gl}(\ell)$ (or ${\rm GL}(\ell)$) is
determined by its highest weight $\la\in H^*$ with $\langle \la,
e_{ii}\rangle=\la_i\in\Z$ $(1\le i\le \ell)$ and $\la_1\geq\ldots
\geq \la_{\ell}$. Denote by $V^\la_{{\rm GL}(\ell)}$ the irreducible
representation corresponding to $\la$. Identifying $\la$ with
$(\la_1,\ldots,\la_{\ell})$ the irreducible rational representations
of ${\rm GL}(\ell)$ are parameterized by
\begin{equation*}
{\mc P}({\rm GL}(\ell)):=\{\,\la=(\la_1,
\cdots,\la_\ell)\,|\,\lambda_i\in\Z, \ \la_1\geq\ldots \geq
\la_{\ell}\,\},
\end{equation*}
which is precisely the set of generalized partitions of length
$\ell$.

\begin{prop} \label{duality} \cite{F} {\rm (cf.~\cite[Theorem 3.1]{W1})} There exists an action of
$\widehat{\gl}_\infty\times{\rm GL}(\ell)$ on $\F^\ell$.
Furthermore, under this joint action, we have
\begin{equation} \label{eq:dual}
\F^\ell\cong\bigoplus_{\la\in {\mc P}({\rm GL}(\ell))}
L(\hgl_\infty,\La^{\mf a}(\la))\otimes V_{{\rm GL}(\ell)}^\la ,
\end{equation}
where $\Lambda^{\mf a}(\la)=\ell\La^{\mf
a}_0+\sum_{j\in\Z}\la'_j\epsilon_j=\sum_{i=1}^\ell \Lambda^{\mf
a}_{\la_i}$.
\end{prop}

Computing the trace of the operator
$\prod_{n\in\Z}x_n^{E_{nn}}\prod_{i=1}^\ell z_i^{e_{ii}}$ on both
sides of \eqnref{eq:dual}, we obtain the following identity:
\begin{equation}\label{combid-classical1}
\prod_{i=1}^\ell\prod_{n\in\N}(1+x_nz_i)(1+x_{1-n}^{-1}z^{-1}_{i})=\sum_{\la\in
{\mc P}({\rm GL}(\ell))}{\rm ch}L(\hgl_\infty,\Lambda^{\mf
a}(\la)){\rm ch}V^\la_{{\rm GL}(\ell)}.
\end{equation}

\subsubsection{The $(\mathfrak{c}_{\infty},{\rm Sp}(2\ell))$-duality}
Let ${\rm Sp}(2\ell)$ denote the symplectic group, which is the
subgroup of ${\rm GL}(2\ell)$ preserving the non-degenerate
skew-symmetric bilinear form on $\C^{2\ell}$ given by
$$\left(
    \begin{array}{cc}
      0 & J_{\ell} \\
      -J_{\ell} & 0 \\
    \end{array}
  \right),
$$
where $J_\ell$ is the following $\ell\times\ell$ matrix:
\begin{equation}\label{side:diagonal}
J_\ell=\begin{pmatrix}
0&0&\cdots&0&1\\
0&0&\cdots&1&0\\
\vdots&\vdots&\vdots&\vdots&\vdots\\
0&1&\cdots&0&0\\
1&0&\cdots&0&0\\
\end{pmatrix}.
\end{equation}
Let $\mf{sp}(2\ell)$ be the Lie algebra of ${\rm Sp}(2\ell)$. We
take as a Borel subalgebra of $\mf{sp}(2\ell)$ the subalgebra of
upper triangular matrices, and as a Cartan subalgebra $H$ the
subalgebra spanned by $\tilde{e}_{i}=e_{ii}-e_{2\ell+1-i,2\ell+1-i}$
($1\leq i\leq \ell$). A finite-dimensional irreducible
representation of $\mf{sp}(2\ell)$ is determined by its highest
weight $\la\in H^*$ with $\langle \la,
\tilde{e}_{i}\rangle=\la_i\in\Z_+$ $(1\le i\le \ell)$ and
$\la_1\geq\ldots \geq \la_{\ell}$. Furthermore each such
representation lifts to an irreducible representation of ${\rm
Sp}(2\ell)$, which is denoted by $V_{{\rm Sp}(2\ell)}^{\lambda}$. We
identify $\la$ with $(\la_1,\ldots,\la_{\ell})$, and put
\begin{equation*}
{\mc P}({\rm Sp}(2\ell)):=\{\,\la=(\la_1,
\cdots,\la_\ell)\,|\,\lambda_i\in\Z_+, \ \la_1\geq\ldots \geq
\la_{\ell}\,\},
\end{equation*}
which is the set of partitions of length no more than $\ell$.

\begin{prop} \label{duality-c} \cite[Theorem 3.4]{W1}
There exists an action of $\mathfrak{c}_{\infty}\times{\rm
Sp}(2\ell)$ on $\F^\ell$. Furthermore, under this joint action, we
have
\begin{equation} \label{eq:dual-c}
\F^\ell\cong\bigoplus_{\la\in {\mc P}({\rm
Sp}(2\ell))}L(\mathfrak{c}_{\infty},\La^{\mf c}(\la))\otimes V_{{\rm
Sp}(2\ell)}^\la,
\end{equation}
where $\Lambda^{\mf c}(\la)=\ell\La^{\mf c}_0+\sum_{k\geq
1}\la'_k\epsilon_k=(\ell-j)\Lambda^{\mf c}_{0}+\sum_{k=1}^j
\Lambda^{\mf c}_{\la_k}$ and $j$ is the number of non-zero parts of
$\la$.
\end{prop}

Computing the trace of the operator $\prod_{n\in
\N}x_n^{\widetilde{E}_{n}}\prod_{i=1}^\ell z_i^{\tilde{e}_{i}}$ on
both sides of \eqnref{eq:dual-c}, we obtain the following identity:
\begin{equation}\label{combid-classical-c}
\prod_{i=1}^\ell\prod_{n\in\N}(1+x_nz_i)(1+x_{n}z^{-1}_{i})=\sum_{\la\in
{\mc P}({\rm Sp}(2\ell))} {\rm
ch}L(\mathfrak{c}_{\infty},\Lambda^{\mf c}(\la)){\rm ch}V^\la_{{\rm
Sp}(2\ell)}.
\end{equation}

\subsubsection{The $(\mathfrak{d}_{\infty},{\rm O}(m))$-duality}
Let $m$ be a positive integer greater than $1$. We define $\ell$ by
$m=2\ell$ when $m$ is even, and $m=2\ell+1$ when $m$ is odd. Let
${\rm O}(m)$ denote the orthogonal group which is the subgroup of
${\rm GL}(m)$ preserving  the non-degenerate symmetric bilinear form
on $\C^{m}$ determined by $J_m$ of \eqnref{side:diagonal}. Let
$\mf{so}(m)$ be the Lie algebra of ${\rm O}(m)$. We take as a Cartan
subalgebra $H$ of $\mf{so}(m)$ the subalgebra spanned by
$\tilde{e}_{i}=e_{ii}-e_{m+1-i,m+1-i}$ ($1\leq i\leq \ell$), while
we take as the Borel subalgebra the subalgebra of upper triangular
matrices.

For $\la\in H^*$ let us identify $\la$ with
$(\la_1,\ldots,\la_{\ell})$, where
$\la_i=\langle\la,\tilde{e}_{i}\rangle$ for $1\le i\le \ell$. Then a
finite-dimensional irreducible representation of $\mf{so}(2\ell)$ is
determined by its highest weight $\la$ satisfying the condition
$\la_1\geq \ldots\geq \la_{\ell-1}\geq |\la_{\ell}|$ with either
$\la_i\in\Z$ or else $\la_i\in\hf+\Z$, for $1\le i\le \ell$.
Furthermore it lifts to a representation of ${\rm SO}(2\ell)$ if and
only if $\la_i\in\Z$ for $1\le i\le \ell$. Also a finite-dimensional
irreducible representation of $\mf{so}(2\ell+1)$ is determined by
its highest weight $\la$ satisfying the conditions $\la_1\geq
\ldots\geq \la_{\ell}$ with either  $\la_i\in\Z_+$ or else
$\la_i\in\hf+\Z_+$, for $1\le i\le \ell$. Furthermore it lifts to a
representation of ${\rm SO}(2\ell+1)$ if and only if $\la_i\in\Z_+$
for $1\le i\le \ell$.

We put
\begin{equation*}
{\mc P}({\rm O}(m)):=\{\,\la=(\la_1,
\cdots,\la_m)\,|\,\lambda_i\in\Z_+, \ \la_1\geq\ldots \geq \la_{m},\
\la'_1+\la'_2\leq m\,\}.
\end{equation*}
For $\la\in {\mc P}({\rm O}(m))$, let $\tilde{\la}$  be the
partition obtained from $\la$ by replacing its first column with
$m-\la'_1$.

Suppose that $m=2\ell$ and
$\lambda=(\la_1,\ldots,\la_{\ell},0,\ldots,0)\in {\mc P}({\rm
O}(2\ell))$ is given. If  $\la_{\ell}>0$, let $V_{{\rm
O}(2\ell)}^{\lambda}$ be the irreducible ${\rm O}(2\ell)$-module,
which as an $\mf{so}(2\ell)$-module, is isomorphic to the direct sum
of irreducible representations of highest weights
$(\la_1,\ldots,\la_{\ell})$ and $(\la_1,\ldots,-\la_{\ell})$. If
$\la_{\ell}=0$, let $V_{{\rm O}(2\ell)}^{\lambda}$ denote the ${\rm
O}(2\ell)$-module that as an $\mf{so}(2\ell)$-module is isomorphic
to the irreducible representation of highest weight
$(\la_1,\cdots,\la_{\ell-1},0)$, and on which the element
$\tau=\sum_{i\not=\ell,\ell+1}{e_{ii}}+e_{\ell,\ell+1}+e_{\ell+1,\ell}\in
{\rm O}(2\ell)\setminus {\rm SO}(2\ell)$ transforms trivially on
highest weight vectors. Set $V_{{\rm
O}(2\ell)}^{\tilde{\lambda}}=V_{{\rm O}(2\ell)}^{\lambda}\otimes
{\rm det}$, where ${\rm det}$ is the one-dimensional non-trivial
representation of ${\rm O}(2\ell)$.

Suppose that $m=2\ell+1$ and
$\lambda=(\la_1,\ldots,\la_{\ell},0,\ldots,0)\in {\mc P}({\rm
O}(m))$ is given. Let $V_{{\rm O}(2\ell+1)}^{\lambda}$ be the
irreducible ${\rm O}(2\ell+1)$-module isomorphic  to the irreducible
representation of highest weight $(\la_1,\ldots,\la_{\ell})$ as an
$\mf{so}(2\ell+1)$-module, on which $-I_{m}$ acts trivially. Here
$I_m$ is the $m\times m$ identity matrix. Also, we let $V_{{\rm
O}(2\ell+1)}^{\tilde{\lambda}}=V_{{\rm O}(2\ell+1)}^{\lambda}\otimes
{\rm det}$ (cf. e.g.~\cite{BT, H} for more details).

\begin{prop} \label{duality-d} \cite[Theorems 3.2 and 4.1]{W1} There exists an action of $\mathfrak{d}_{\infty}\times{\rm
O}(m)$ on $\F^{\frac{m}{2}}$. Furthermore under this joint action we
have
\begin{equation} \label{eq:dual-d}
\F^{\frac{m}{2}}\cong\bigoplus_{\la\in {\mc P}({\rm
O}(m))}L(\mathfrak{d}_{\infty},\La^{\mf d}(\la))\otimes V_{{\rm
O}(m)}^\la,
\end{equation}
where $\Lambda^{\mf d}(\la)=m\Lambda^{\mf d}_{0}+\sum_{k\geq 1}
\la'_k\epsilon_k$.
\end{prop}

Suppose that $m=2\ell$. Computing the trace of the operator
$\prod_{n\in \N}x_n^{\widetilde{E}_{n}}\prod_{i=1}^\ell
z_i^{\tilde{e}_{i}}$ on both sides of \eqnref{eq:dual-d}, we obtain
\begin{equation}\label{combid-classical-d1}
\prod_{i=1}^\ell\prod_{n\in\N}(1+x_nz_i)(1+x_{n}z^{-1}_{i})=\sum_{\la\in
{\mc P}({\rm O}(2\ell))} {\rm
ch}L(\mathfrak{d}_{\infty},\Lambda^{\mf d}(\la)){\rm ch}V^\la_{{\rm
O}(2\ell)}.
\end{equation}

Suppose that $m=2\ell+1$. Let $\epsilon$ be the eigenvalue of $-I_m$
on ${\rm O}(2\ell+1)$-modules satisfying $\epsilon^2=1$. From the
computation of the trace of $\prod_{n\in
\N}x_n^{\widetilde{E}_{n}}\prod_{i=1}^\ell
z_i^{\tilde{e}_{i}}(-I_m)$ on both sides of \eqnref{eq:dual-d}, we
obtain
\begin{equation}\label{combid-classical-d2}
\prod_{i=1}^\ell\prod_{n\in\N}(1+\epsilon x_nz_i)(1+ \epsilon
x_{n}z^{-1}_{i})(1+\epsilon x_n)=\sum_{\la\in {\mc P}({\rm
O}(2\ell+1))} {\rm ch}L(\mathfrak{d}_{\infty},\Lambda^{\mf
d}(\la)){\rm ch}V^\la_{{\rm O}(2\ell+1)}.
\end{equation}
Note that ${\rm ch}V^\la_{{\rm O}(2\ell)}$ is a Laurent polynomial
in $z_1,\ldots,z_{\ell}$ and ${\rm ch}V^\la_{{\rm O}(2\ell)}={\rm
ch}V^{\tilde{\la}}_{{\rm O}(2\ell)}$, while ${\rm ch}V^\la_{{\rm
O}(2\ell+1)}$ is the Laurent polynomial in
$z_1,\ldots,z_{\ell},\epsilon$ and ${\rm ch}V^\la_{{\rm
O}(2\ell+1)}=\epsilon~{\rm ch}V^{\tilde{\la}}_{{\rm O}(2\ell+1)}$.

\subsubsection{The $(\mathfrak{b}_{\infty},{\rm Pin}(2\ell))$-duality}

The Lie group ${\rm Pin}(2\ell)$ is a double cover of ${\rm
O}(2\ell)$, with ${\rm Spin}(2\ell)$ as the inverse image of ${\rm
SO}(2\ell)$ under the covering map (see e.g.~\cite{BT}). An
irreducible representation of ${\rm Spin}(2\ell)$ that does not
factor through ${\rm SO}(2\ell)$ is an irreducible representation of
$\mf{so}(2\ell)$ of highest weight of the form
\begin{equation}\label{hwforSpin}
(\la_1+\hf,\ldots,\la_{\ell-1}+\hf,\la_{\ell}+\hf), \ \ \text{or} \
\ (\la_1+\hf,\ldots,\la_{\ell-1}+\hf,-\la_{\ell}-\hf),
\end{equation}
where $\la_1\geq \ldots\geq \la_{\ell}$ with $\la_i\in\Z_+$ for
$1\le i\le \ell$. We put
\begin{equation*}
{\mc P}({\rm Pin}(2\ell)):=\{\,\la=(\la_1,
\cdots,\la_\ell)\,|\,\lambda_i\in\Z_+, \ \la_1\geq\ldots \geq
\la_{\ell}\,\}.
\end{equation*}
For $\lambda\in {\mc P}({\rm Pin}(2\ell))$, let us denote by
$V^{\lambda}_{{\rm Pin}(2\ell)}$ the irreducible representation of
${\rm Pin}(2\ell)$ induced from the irreducible representation of
${\rm Spin}(2\ell)$ whose highest weight is given by either of the
two weights in \eqref{hwforSpin}. When restricted to ${\rm
Spin}(2\ell)$, $V^{\lambda}_{{\rm Pin}(2\ell)}$ decomposes into a
direct sum of two irreducible representations of highest weights
given by those in \eqref{hwforSpin}.

\begin{prop} \label{duality-b} \cite[Theorem 3.3]{W1} There exists an action of $\mathfrak{b}_{\infty}\times{\rm
Pin}(2\ell)$ on $\F^{\ell}$. Furthermore, under this joint action,
we have
\begin{equation} \label{eq:dual-b}
\F^{\ell}\cong\bigoplus_{\la\in {\mc P}({\rm
Pin}(2\ell))}L(\mathfrak{b}_{\infty},\La^{\mf b}(\la))\otimes
V_{{\rm Pin}(2\ell)}^\la,
\end{equation}
where $\Lambda^{\mf b}(\la)=2\ell\Lambda^{\mf b}_{0}+\sum_{k\geq 1}
\la'_k\epsilon_k$.
\end{prop}

Taking the trace of the operator $\prod_{n\in
\N}x_n^{\widetilde{E}_{n}}\prod_{i=1}^\ell z_i^{\tilde{e}_{i}}$ on
both sides of \eqnref{eq:dual-b}, gives
\begin{equation}\label{combid-classical-b}
\prod_{i=1}^\ell\prod_{n\in\N}(z_i^{\hf}+z_i^{-\hf})(1+x_nz_i)(1+x_{n}z^{-1}_{i})=\sum_{\la\in
{\mc P}({\rm Pin}(2\ell))} {\rm
ch}L(\mathfrak{b}_{\infty},\Lambda^{\mf b}(\la)){\rm ch}V^\la_{{\rm
Pin}(2\ell)}.
\end{equation}

In what follows we mean by $(\mf{g},G)$ one of the dual pairs of
\secref{classical:dualpairs}, and by $\mf{x}\in\{\mf{a,b,c,d}\}$ the
type of $\G$.

\subsection{$\mathfrak u_-$-homology groups of ${\mf
g}$-modules}\label{classical:homology} Let
$\Delta:=\Delta^+\cup\Delta^-$ be the set of roots of $\G$, where
$\Delta^-=-\Delta^+$. Let
$\Delta_S^\pm:=\Delta^\pm\cap\big{(}\sum_{j\not=0}\Z
\alpha_j\big{)}$ and
$\Delta^\pm(S):=\Delta^\pm\setminus\Delta^\pm_S$. Denote by
$\G_\alpha$ the root space corresponding to $\alpha\in\Delta$. Set
\begin{equation}\label{parabolic}
\begin{aligned}
&{\mf u}_{\pm} := \sum_{\alpha\in \Delta^\pm(S)}\G_{\alpha}, \quad
{\mf l}  := \sum_{\alpha\in
\Delta_S^+\cup\Delta_S^-}\G_{\alpha}\oplus \h.
\end{aligned}
\end{equation}
Then we have $\G=\mathfrak u_+\oplus\mathfrak l\oplus\mathfrak u_-$.
The Lie algebras $\mathfrak l$ and $\G$ share the same Cartan
subalgebra $\h$. For $\mu\in\h^*$ we denote  by $L(\mathfrak l,\mu)$
the irreducible highest weight representation of $\mathfrak l$ with
highest weight $\mu$. For $\G=\hgl_\infty$ we denote by $\mathcal
P^+_{\mathfrak l}$ the set of
$\mu=\sum_{i\in\Z}\mu_i\epsilon_i+c\La^{\mf a}_0\in\h^*$ with
$(\mu_1,\mu_2,\cdots)\in\mc{P}^+$,
$(-\mu_{0},-\mu_{-1},\cdots)\in\mc{P}^+$, and $c\in\C$. For
$\G=\mf{b}_\infty,\mf{c}_\infty, \mf{d}_\infty$, let
$\mc{P}^+_{\mathfrak l}$ denote the set of
$\mu=\sum_{i\in\N}\mu_i\epsilon_i+c\La^x_0\in\h^*$ with
$(\mu_1,\mu_2,\cdots)\in\mc{P}^+$ and $c\in\C$.

Let $V$ be a $\G$-module and let $C_k(\mathfrak
u_-;V):=\Lambda^k(\mathfrak u_-)\otimes V$ be the space of the
$k$-th chains ($k\in\Z_+$). Denote the boundary operator by
$d_k:C_k(\mathfrak u_-;V)\rightarrow C_{k-1}(\mathfrak u_-;V)$ (see
e.g. \cite{GL,J,L}). We have a complex
\begin{equation}\label{chains}
\cdots\stackrel{d_{k+1}}{\longrightarrow} C_k(\mathfrak
u_-;V)\stackrel{d_k}{\longrightarrow}C_{k-1}(\mathfrak
u_-;V)\stackrel{d_{k-1}}{\longrightarrow}\cdots\stackrel{d_2}{\longrightarrow}\mathfrak
u_-\otimes
V\stackrel{d_1}{\longrightarrow}V\stackrel{d_0}{\longrightarrow} 0,
\end{equation}
with $d_kd_{k+1}=0$ and homology groups ${\rm H}_k(\mathfrak
u_-;V):={\rm ker}\,d_k/{\rm im}\,d_{k+1}$ for $k\in\Z_+$.

The subalgebra $\mathfrak l$ acts on $V$ by restriction, while it
acts on $\mathfrak u_-$ via the adjoint action.  Thus $\mathfrak l$
acts on the chains.  Furthermore the $\mathfrak l$-action commutes
with boundary operators and hence the homology group ${\rm
H}_k(\mathfrak u_-;V)$ is an $\mathfrak l$-module, for all
$k\in\Z_+$. In order to describe these groups in the case when
$V=L(\G,\La^\mf{x}(\la))$, for $\la\in \mathcal{P}(G)$, we introduce
further notation.

For $j\in I$, define simple reflections $\sigma_j$ by
\begin{equation*}
\sigma_j(\mu):=\mu-\langle\mu,{\alpha}^{\vee}_j\rangle\alpha_j,
\end{equation*}
where $\mu\in\h^*$. Let $W$ be the subgroup of ${\rm Aut}(\h^*)$
generated by the simple reflections, i.e. $W$ is the Weyl group of
$\G$. For each $w\in W$ we let $l(w)$ denote the length of $w$. We
have an action on $\h$ given by
$\sigma_j(h)=h-\langle\alpha_j,h\rangle {\alpha}^{\vee}_j$ for
$h\in\h$ so that $\langle w(\mu),w(h)\rangle=\langle\mu,h\rangle$,
for $w\in W$, $\mu\in\h^*$ and $h\in\h$. We also define
\begin{equation*}
w\circ\mu:=w(\mu+\rho_c)-\rho_c, \quad\mu\in\h^*,w\in W.
\end{equation*}
Consider $W_{0}$ the subgroup of $W$ generated by $\sigma_j$ with
$j\not=0$. Let
\begin{equation*}
W^{0}:=\{\,w\in W\,\vert\,
w(\Delta^-)\cap\Delta^+\subseteq\Delta^+(S)\,\}.
\end{equation*}
It is well-known that $W=W_0\, W^0$ and $W^0$ is the set of the
minimal length representatives of the right coset space
$W_0\backslash W$ (cf.~\cite{L}). For $k\in\Z_+$, set
$$W^0_k:=\{\,w\in W^0\,\vert\, l(w)=k\,\}.$$

Given $\lambda\in \mc{P}(G)$ it is easy to see that
$\langle\Lambda^\mf{x}(\la),{\alpha}^{\vee}_j\rangle\in\Z_+$, for
all $j\in I$. Since $w\in W^0$ implies that
$w^{-1}(\Delta^+_S)\subseteq\Delta^+$, we obtain $\langle w\circ
\Lambda^\mf{x}(\la),{\alpha}^{\vee}_j\rangle\in\Z_+$, for all $j\in
I\setminus\{0\}$. Since $\G$ is a generalized Kac-Moody algebra and
each $L(\G,\Lambda^\mf{x}(\la))$ is an integrable (=standard)
module, the description of the $\mathfrak l$-module structure of the
homology groups ${\rm H}_k({\mathfrak
u}_-;L(\G,\Lambda^\mf{x}(\la))$ in \cite{J} remains valid. It is
given as follows.

\begin{prop}\label{homologyclassical}\cite{J}
Let $\la\in\mc{P}(G)$ and $k\in\Z_+$.  We have, as
$\mathfrak{l}$-modules,
\begin{equation*}
{\rm H}_k({\mathfrak
u}_-;L(\G,\Lambda^\mf{x}(\la)))\cong\bigoplus_{w\in
W^0_k}L(\mathfrak l\,,w\circ \Lambda^\mf{x}(\la)).
\end{equation*}
\end{prop}

We may write $w\circ\La^\mf{x}(\la)$ as
\begin{equation*}
w\circ\La^\mf{x}(\la)=
\begin{cases}
\sum_{j>0}(\la^+_w)_j\epsilon_j-\sum_{j\ge
0}(\la^-_w)_{j+1}\epsilon_{-j}+n\Lambda^{\mf a}_0, & \text{if $\mf{x}={\mf a}$}, \\
\sum_{j>0}(\la_w)_j\epsilon_j+n\Lambda^\mf{x}_0, & \text{otherwise}, \\
\end{cases}
\end{equation*}
where
$n=\langle\,\La^\mf{x}(\la),K\,\rangle/\langle\,\La^\mf{x}_0,K\,\rangle$.
Furthermore, $\la^{\pm}_w=((\la^{\pm}_w)_1,(\la^{\pm}_w)_2,\ldots)$
and $\la_w=((\la_w)_1,(\la_w)_2,\ldots)$ are partitions. Thus,
suppressing the action of $K$, we see that the character of the
$\mathfrak l$-module ${\rm H}_k({\mathfrak
u}_-;L(\G,\La^\mf{x}(\la)))$ is
\begin{equation}\label{homology:schur}
{\rm ch H}_k({\mathfrak u}_-;L(\G,\La^\mf{x}(\la)))=
\begin{cases}
\sum_{w\in W^0_k}s_{\la^+_w}(x_1,x_2,\cdots)
s_{\la^-_w}(x^{-1}_0,x^{-1}_{-1},\cdots), & \text{if
$\mf{x}={\mf a}$}, \\
\sum_{w\in W^0_k}s_{\la_w}(x_1,x_2,\cdots), & \text{otherwise},
\end{cases}
\end{equation}
where here and further $s_\mu(y_1,y_2,\cdots)$ denotes the Schur
function in the variables $y_1,y_2,\cdots$ associated to the
partition $\mu$. Now applying the Euler-Poincar\'e principle to
\eqnref{chains} one obtains
\begin{equation*}
\sum_{k=0}^\infty (-1)^k {\rm ch}C_k(\mathfrak
u_-;L(\G,\La^\mf{x}(\la)))= \sum_{k=0}^\infty (-1)^k {\rm ch
H}_k(\mathfrak u_-;L(\G,\La^\mf{x}(\la))).
\end{equation*}
Since $$\sum_{k=0}^\infty (-1)^k {\rm ch}C_k(\mathfrak
u_-;L(\G,\La^\mf{x}(\la)))={\rm
ch}L(\G,\La^\mf{x}(\la)){D^\mf{x}},$$ where
\begin{equation*}
D^\mf{x}=
\begin{cases}
\prod_{i,j}(1-x_{-i+1}^{-1}x_j), & \text{if
$\mf{x}={\mf a}$}, \\
\prod_{i}(1-x_i)\prod_{i< j}(1-x_ix_j), & \text{if $\mf{x}=\mf{b}$}, \\
\prod_{i}(1-x^2_i)\prod_{i< j}(1-x_ix_j), & \text{if
$\mf{x}=\mf{c}$},
\\ \prod_{i< j}(1-x_ix_j), & \text{if $\mf{x}=\mf{d}$},
\end{cases}
\end{equation*}
with $i,j\in\N$, we obtain
\begin{equation}\label{char:schur}
{\rm ch}L(\G,\La^\mf{x}(\la))=\frac{1}{D^\mf{x}}\sum_{k=0}^\infty
(-1)^k {\rm ch H}_k(\mathfrak u_-;L(\G,\La^\mf{x}(\la))).
\end{equation}

\section{Infinite-dimensional Lie superalgebras}\label{IDLSA}

\subsection{The Lie superalgebra $\widehat{\gl}_\dinfty$ and its subalgebras of classical
types}\label{ex:superalgebras}

Let us recall the infinite-dimensional Lie superalgebra
$\widehat{\gl}_{\infty|\infty}$ and the subalgebras
$\widehat{\mc{B}}$, $\widehat{\mc{C}}$, $\widehat{\mc{D}}$ of
classical types. We will say that these superalgebras are of types
$\mf{a,b,c,d}$, respectively. The following notation will also be
assumed throughout the paper.

\begin{itemize}
\item[$\cdot$] $\overline{\mathfrak{g}}$ : the Lie
superalgebra $\widehat{\gl}_{\infty|\infty}$, or  its subalgebra
$\widehat{\mc{B}}$, $\widehat{\mc{C}}$, $\widehat{\mc{D}}$,

\item[$\cdot$] $\overline{\mathfrak{h}}$ : a Cartan subalgebra of
$\overline{\mathfrak{g}}$,

\item[$\cdot$] $\overline{I}$ : the index set for simple roots,

\item[$\cdot$] $\overline{\Pi}=\{\,\beta_r\,|\,r\in \overline{I}\,\}$ : the set of
simple roots,

\item[$\cdot$] $\overline{\Pi}^{\vee}=\{\,\beta_r^{\vee}\,|\,r\in \overline{I}\,\}$ : the set of
simple coroots,

\item[$\cdot$] $\overline{\Delta}^{+}$ : the set of positive roots.



\end{itemize}

\subsubsection{The Lie superalgebra
$\widehat{\gl}_\dinfty$}\label{aux311}

Let $\C^{\dinfty}$ be the infinite-dimensional superspace over $\C$
with a basis $\{\,e_r\,\vert\, r\in\hf\Z\,\}$. We assume that ${\rm
deg}e_r=\bar{0}$, for $r\in\Z$, and ${\rm deg}e_r=\bar{1}$,
otherwise.  We may identify ${\rm End}(\C^{\infty|\infty})$ with the
Lie superalgebra of matrices $(a_{rs})$  ($r,s\in\hf\Z$). Let
$\gl_\dinfty$ denote the subalgebra consisting of $(a_{rs})$  with
$a_{rs}=0$ for all but finitely many $a_{rs}$'s. Denote by $E_{rs}$
the elementary matrix with $1$ at the $r$-th row and $s$-th column
and zero elsewhere. Denote by
$\widehat{\gl}_\dinfty=\gl_\dinfty\oplus\C K$ the central extension
of $\gl_\dinfty$ by a one-dimensional center $\C K$ given by the
$2$-cocycle
\begin{equation}\label{cocycle-super}
\beta(A,B):={\rm Str}([\bar{J},A]B),
\end{equation}
where $\bar{J}=\sum_{r\le 0}E_{rr}$. Here, for a matrix
$D=(d_{rs})$, the supertrace is defined by ${\rm
Str}D=\sum_{r\in\hf\Z}(-1)^{2r}d_{rr}$. Then
$\overline{\h}=\sum_{r\in\hf\Z}\C E_{rr}\oplus\C K$ is a Cartan
subalgebra. Put $\overline{I}=\hf\Z$. We have
\begin{align*}
\overline{\Pi}^{\vee}=\{\,& \beta_r^{\vee}=E_{rr}+E_{r+\hf,r+\hf}+\delta_{r0}K\ (r\in \overline{I}) \, \}, \\
\overline{\Pi}=\{\,& \beta_r=\delta_r-\delta_{r+\hf} \ (r\in \overline{I})\, \}, \\
\overline{\Delta}^+ =\{\,& \delta_r-\delta_s\ (r,s\in\ov{I}, r<s)
\,\},
\end{align*}
where $\delta_r\in\overline{\h}^*$ is determined by $\langle
\delta_r,E_{ss}\rangle=\delta_{rs}$ and
$\langle\delta_r,K\rangle=0$.

By assigning degree $0$ to the Cartan subalgebra and setting ${\rm
deg}E_{rs}=s-r$, we equip $\hgl_\dinfty$ with a $\hf\Z$-gradation:
$\hgl_\dinfty=\bigoplus_{k\in\hf\Z}(\hgl_\dinfty)_k$. This leads to
the triangular decomposition:
\begin{equation*}
\hgl_\dinfty=(\hgl_\dinfty)_+\oplus (\hgl_\dinfty)_0\oplus
(\hgl_\dinfty)_-,
\end{equation*}
where
$(\hgl_\dinfty)_{\pm}=\bigoplus_{k\in\pm\hf\N}(\hgl_\dinfty)_k$ and
$(\hgl_\dinfty)_0=\overline{\h}$.

\subsubsection{The Lie superalgebras $\widehat{\mc{B}}$, $\widehat{\mc{C}}$, $\widehat{\mc{D}}$}
Below we follow the presentations of \cite{CW2,LZ}. Let $L=L_{\ov
0}\oplus L_{\ov 1}$ be a Lie superalgebra, which is a subalgebra of
$\gl_{\dinfty}$. We say that $L$ preserves a bilinear form
$(\,\cdot\,|\,\cdot\,)$ on $\C^{\dinfty}$ if
\begin{equation*}
L_{\epsilon}=\{\,A\in
(\hgl_{\dinfty})_{\epsilon}\,|\,(Av|w)=-(-1)^{\epsilon|v|}(v|Aw),\
v,w\in \C^{\dinfty}\,\}, \quad \epsilon=\bar{0},\bar{1},
\end{equation*}
where $|v|$ denotes the parity of $v$ in $\C^{\dinfty}$
(cf.~\cite{K1}).

First we define the Lie superalgebra $\mc{B}$ to be the subalgebra
of $\gl_{\dinfty}$ preserving the following super-symmetric bilinear
form on $\C^{\dinfty}$:
\begin{align*}
(e_i|e_j)&= (-1)^i\delta_{i,-j}, \quad\quad i,j\in\Z, \\
(e_r|e_s)&= (-1)^{r+\hf}\delta_{r,-s}, \quad\quad r,s\in\hf+\Z, \\
(e_i|e_r)&= 0, \quad\quad i\in\Z,\ r\in \hf+\Z.
\end{align*}

Let $\mc{A}$ be the subalgebra of $\mf{gl}_{\dinfty}$ consisting of
matrices $A=(a_{ij})$ such that $a_{ij}=0$ if $i=0$ or $j=0$. We
define the Lie superalgebras $\mc{C}$ and $\mc{D}$ to be the
subalgebras of $\mc{A}$ preserving the following
super-skew-symmetric bilinear forms on $\C^{\dinfty}$, respectively:
\begin{align*}
(e_i|e_j)&=
\begin{cases}
{\rm sgn}(i)\delta_{i,-j},   & \text{for $\mc{C}$}, \\
\delta_{i,-j}, & \text{for $\mc{D}$},
\end{cases} \quad\quad i,j\in\Z\setminus\{0\}, \\
(e_r|e_s)&=
\begin{cases}
\delta_{r,-s},   & \text{for $\mc{C}$}, \\
{\rm sgn}(r)\delta_{r,-s}, & \text{for $\mc{D}$},
\end{cases}  \quad\quad r,s\in\hf+\Z, \\
(e_i|e_r)&= 0, \quad\quad i\in\Z,\ r\in \hf+\Z.
\end{align*}

Now let $\mc{X}$ be one of $\mc{B,C,D}$. We define
$\widehat{\mc{X}}$ to be the central extension of $\mc{X}$ given by
the restriction of the two-cocycle \eqnref{cocycle-super}. Then
$\widehat{\mc{X}}$ has a natural triangular decomposition induced
from $\hgl_{\dinfty}$:
\begin{equation*}
\widehat{\mc{X}}=\widehat{\mc{X}}_+\oplus \widehat{\mc{X}}_0\oplus
\widehat{\mc{X}}_-,
\end{equation*}
where
$\widehat{\mc{X}}_0=\overline{\mathfrak{h}}=\sum_{r\in\hf\N}\C\widetilde{E}_r\oplus
\C K$ is a Cartan subalgebra, with
$\widetilde{E}_r=E_{rr}-E_{-r,-r}$ for $r\in \hf\N$. We may regard
$\delta_r\in (\hgl_{\dinfty})_0^*$ as an element in
$\widehat{\mc{X}}_0^*$ via restriction so that
$\langle\delta_r,\widetilde{E}_s\rangle=\delta_{rs}$ for $r,s\in
\hf\mathbb{N}$. Put $\overline{I}=\hf\Z_{+}$. Then we have\vskip 3mm

$\bullet$ $\widehat{\mc{B}}$
\begin{align*}
\ov{\Pi}^{\vee}=\{\,& \beta_0^{\vee}=\widetilde{E}_{\hf}+K, \
  \beta_r^{\vee}=\widetilde{E}_r+\widetilde{E}_{r+\hf} \ (r\in\frac{1}{2}\N)\, \}, \\
\ov{\Pi}=\{\,& \beta_0=-\delta_{\hf}, \
 \beta_r=\delta_r-\delta_{r+\hf} \ (r\in\hf\N)\, \}, \\
\ov{\Delta}^+ =\{\,& -\delta_r\ (r\in\hf\N),-2\delta_r\
(r\in\hf+\Z_+),\ \pm\delta_r-\delta_s\ (r,s\in\hf\N, \ r<s ) \,\}.
\end{align*}

$\bullet$ $\widehat{\mc{C}}$
\begin{align*}
\ov{\Pi}^{\vee}=\{\,&
\beta_0^{\vee}=\widetilde{E}_{\hf}-\widetilde{E}_{1}+2K, \
  \beta_r^{\vee}=\widetilde{E}_r+\widetilde{E}_{r+\hf} \ (r\in\hf\N)\, \}, \\
\ov{\Pi}=\{\,& \beta_0=-\delta_{\hf}-\delta_1, \
 \beta_r=\delta_r-\delta_{r+\hf} \ (r\in\hf\N)\, \}, \\
\ov{\Delta}^+ =\{\,& -2\delta_n\ (n\in\N),\ \pm\delta_r-\delta_s\
(r,s\in\hf\N, \ r<s ) \,\}.
\end{align*}

$\bullet$ $\widehat{\mc{D}}$
\begin{align*}
\ov{\Pi}^{\vee}=\{\,& \beta_0^{\vee}=\widetilde{E}_{\hf}+K, \
  \beta_r^{\vee}=\widetilde{E}_r+\widetilde{E}_{r+\hf} \ (r\in\hf\N)\, \}, \\
\ov{\Pi}=\{\,& \beta_0=-2\delta_{\hf}, \
 \beta_r=\delta_r-\delta_{r+\hf} \ (r\in\hf\N)\, \}, \\
\ov{\Delta}^+ =\{\,& -2\delta_r\ (r\in\hf\N),\ \pm\delta_r-\delta_s\
(r,s\in\hf\N, \ r<s ) \,\}.
\end{align*}

\subsection{Tensor representations of $\gl(m|n)$ and hook Schur
functions}\label{tensor:repn}

\subsubsection{Hook Schur functions}
Let $\la\in\mathcal P^+$ and let $s_\la(x_1,y_1,x_2,y_2,\cdots)$ be
the Schur function corresponding to $\la$.  We can write (see
e.g.~\cite{M})
\begin{equation*}
s_\la(x_1,y_1,x_2,y_2,x_3,y_3,\cdots)=\sum_{\mu\subseteq\la}s_{\mu}(x_1,x_2,\cdots)s_{\la/\mu}(y_1,y_2,\cdots).
\end{equation*}
Define the hook Schur function corresponding to $\la$ \cite{BR,S} to
be
\begin{equation*}
HS_\la(x_1,y_1,x_2,y_2,\cdots)=\sum_{\mu\subseteq\la}s_{\mu}(x_1,x_2,\cdots)s_{(\la/\mu)'}(y_1,y_2,\cdots).
\end{equation*}
From the definition it is not difficult to check the following (cf.~
\cite[Chapter I $\S$5 Ex.~23]{M}).
\begin{lem}\label{aux:hookSchur} For a partition $\la$ we have
\begin{equation*}
HS_{\la'}(x_1,y_1,x_2,y_2,\cdots)=HS_{\la}(y_1,x_1,y_2,x_2,\cdots).
\end{equation*}
\end{lem}

\subsubsection{Irreducible tensor representations of
$\gl(m|n)$}\label{standardborel}

Below we present some basic results for the Lie superalgebra
$\gl(m|n)$, for $m\in \N$ and $n\in\N\cup\{\infty\}$, that will be
used in the sequel.

Consider the space $\C^{m|n}$ spanned by basis elements $e_i$ with
$i\in\{-m,\cdots,-1\}\cup\{1,2,\cdots,n\}$, for $n$ finite, and
$i\in\{-m,\cdots,-1\}\cup\N$, for $n=\infty$.  We assume that ${\rm
deg}e_i=\bar{0}$ for $i<0$, and ${\rm deg}e_i=\bar{1}$ for $i>0$.
The Lie superalgebra $\gl(m|n)$ is defined similarly as in
\secref{aux311}. Denote the elementary matrices by $\{E_{ij}\}$.
Here in \secref{standardborel} our Borel subalgebra is $\sum_{i\le
j}\C E_{ij}$ with Cartan subalgebra $\overline{\h}=\sum_{j}\C
E_{jj}$. Let $\delta_i\in\overline{\h}^*$ be defined by
$\langle\delta_{i},E_{jj}\rangle=\delta_{ij}$.  We choose a
symmetric bilinear form $(\cdot|\cdot)_s$ on $\overline{\h}^*$ such
that $(\delta_i|\delta_j)_s=-{\rm sgn}(i)\delta_{ij}$. For
$\mu\in\ov{\h}^*$ we denote by $L(\gl(m|n),\mu)$ the irreducible
highest weight module of highest weight $\mu$.

In this paragraph let us assume that $n$ is finite.  It is
well-known \cite{BR, S} that the tensor powers of the standard
module $\C^{m|n}$ are completely reducible. The highest weights of
the irreducible $\gl(m|n)$-modules that appear as irreducible
components in a tensor power of $\C^{m|n}$ are parameterized by the
set
\begin{equation*}
\mc{P}^+_{m|n}:=\{(\la_{-m},\la_{-m+1},\cdots,\la_{-1},\la_1,\la_2,\cdots)\in\mc{P}^+\vert\la_1\le
n\}.
\end{equation*}
To be more precise, let us write
$\la^{-}=(\la_{-m},\cdots,\la_{-1})$, $\la^{+}=(\la_1,\la_2,\cdots)$
and $\la=(\la^{-}|\la^{+})$.  The irreducible $\gl(m|n)$-modules
that appear in $(\C^{m|n})^{\otimes k}$ have highest weights
precisely of the form
\begin{equation}\label{lambda:transpose}
\la^\natural=\sum_{i=-m}^{-1}\la_i\delta_i+\sum_{j\ge
1}\la'_j\delta_j,\quad|\la|=k.
\end{equation}
Such a module $L(\gl(m|n),\la^\natural)$ will be called an {\em
irreducible tensor module}. Furthermore, putting $e^{\delta_i}=x_i$
($i<0$) and $y_j=e^{\delta_j}$ ($j>0$) the character of
$L(\gl(m|n),\la^\natural)$ is given by \cite{BR, S}
\begin{equation}\label{hookschur}
\sum_{\mu\subseteq\la}s_\mu(x_{-m},\cdots,x_{-1})s_{(\la/\mu)'}(y_1,y_2,\cdots).
\end{equation}

For $n\in\N\cup\{\infty\}$ let $X^{m|n}$ denote the set of sequences
of $m+n$ integers, which we index by the set
$\{-m,\cdots,-1,1,\cdots,n\}$ for $n\in\mathbb{N}$, and
$\{-m,\cdots,-1\}\cup\mathbb{N}$ for $n=\infty$. We assume that a
sequence in $X^{m|n}$ has only finitely many non-zero entries. Let
$X_+^{m|n}\subseteq X^{m|n}$ be the subset consisting of sequences
of the form $\mu=(\mu^-|\mu^+)$, where
$\mu^-=(\mu_{-m},\cdots,\mu_{-1})$ is a generalized partition of
length $m$ and $\mu^{+}=(\mu_1,\mu_2,\ldots)$ is a partition. We may
regard $\mu\in X^{m|n}$ as the element $\sum_{i\geq -m,~ i\neq
0}\mu_i\delta_i\in\ov{\h}^*$.

Consider the category $\mathcal O^+_{m|n}$ of $\ov{\h}$-semisimple
$\gl(m|n)$-modules consisting of objects $M$ that have composition
series of the form
\begin{equation*}
0\subseteq M_1\subseteq M_2\subseteq M_3\subseteq\ldots
\end{equation*}
with $\bigcup_{i=1}^\infty M_i=M$, and such that each composition
factor is of the form $L(\gl(m|n),\mu)$, $\mu\in X_+^{m|n}$.
Furthermore we assume that each isotypic composition factor appears
with finite multiplicity.

Recall the truncation functor of \cite[Definition 4.4]{CWZ} defined
as follows. Let $n$ be finite and let $\mu\in X_+^{m|\infty}$. Note
that in the case $(\mu|\delta_{n+1})_s=0$ we may regard $\mu$ as a
weight of $\gl(m|n)$ by ignoring the zeros. Define for $M\in\mathcal
O^+_{m|\infty}$ a functor $\mf{tr}_{n}:\mathcal
O^+_{m|\infty}\rightarrow\mathcal O^+_{m|n}$ by setting
${\mf{tr}}_n(M)$ to be the linear span of all weight vectors $m\in
M$ such that $({\rm wt}(m)|\delta_j)_s=0$ for $j> n$, where ${\rm
wt}(m)$ denotes the weight of $m$. Clearly $\mf{tr}_n$ is an exact
functor. It is shown in \cite[Corollary 3.6]{CWZ} that
\begin{align}\label{truncation}
{\mf{tr}}_n(L(\gl(m|\infty),\mu))=&
\begin{cases}
L(\gl(m|n),\mu), & \text{ if
}(\mu|\delta_{n+1})_s=0,\\
0, & \text{ otherwise}.
\end{cases}
\end{align}

\begin{lem}\label{finite:to:infinite}
The $\gl(m|\infty)$-module $(\C^{m|\infty})^{\otimes k}$ is
completely reducible.  Furthermore we have
\begin{equation*}
(\C^{m|\infty})^{\otimes k}\cong\bigoplus_{\la}
L(\gl(m|\infty),\la^\natural)^{m_{\la}},
\end{equation*}
where the sum ranges over all partitions
$\la=(\la_{-m},\cdots,\la_{-1},\la_1,\la_2,\cdots)$ with $|\la|=k$,
and $m_{\la}\in\N$.
\end{lem}

\begin{proof}
Suppose that in $(\C^{m|\infty})^{\otimes k}$ we have a non-trivial
extension of the form
\begin{equation}\label{aux:ext}
0\longrightarrow L(\gl(m|\infty),\gamma)\longrightarrow
E\longrightarrow L(\gl(m|\infty),\mu)\longrightarrow 0.
\end{equation}
Take a finite $n\gg 0$ such that
$(\gamma|\delta_{j})_s=(\mu|\delta_j)_s=0$, for $j\ge n+1$.
Applying the truncation functor $\mf{tr}_n$ upon \eqnref{aux:ext} we
get an exact sequence of $\gl(m|n)$-modules
\begin{equation*}
0\longrightarrow L(\gl(m|n),\gamma)\longrightarrow \mf{tr}_n
E\longrightarrow L(\gl(m|n),\mu)\longrightarrow 0.
\end{equation*}
Since $\mf{tr}_n(\C^{m|\infty})^{\otimes k}=(\C^{m|n})^{\otimes k}$,
$\mf{tr}_n E$ is an extension inside $(\C^{m|n})^{\otimes k}$, which
therefore must be a split extension.  Thus we may assume that there
exists $v_{\mu}\in E$ that is annihilated by all $E_{i,i+1}$, for
$i\le n-1$.  Now $E_{i,i+1}v_\mu$ has weight
$\mu+\delta_i-\delta_{i+1}$, for $i>n$.  But there are no such
weight spaces in $E$, being an extension of two highest weight
modules of highest weights $\gamma$ and $\mu$. Thus
$E_{i,i+1}v_\mu=0$, for $i\ge n$, and hence $v_\mu$ is a genuine
highest weight vector inside $E$. Now
$U(\gl(m|k))v_\mu=L(\gl(m|k),\mu)$, for all $k\ge n$, and hence
$U(\gl(m|\infty))v_\mu=L(\gl(m|\infty),\mu)$.
%
%
It follows that $U(\gl(m|\infty))v_\mu\cap
U(\gl(m|\infty))v_\gamma=0$, and thus \eqnref{aux:ext} is a split
extension.

The second statement is easy using the truncation functor again
together with the statement for finite $n$.
\end{proof}

\lemref{finite:to:infinite} allows us to define {\em irreducible
tensor representations} of $\gl(m|\infty)$ in an analogous fashion.
It follows that they are parameterized by $\mc P^+$. Furthermore the
character of $L(\gl(m|\infty), \la^\natural)$ is given by
\eqnref{hookschur} with $n=\infty$.

Assume that $n$ is finite. Let $\gamma\in X^{m|n}$. We denote the
central character corresponding to $\gamma$ by $\chi_\gamma$ (see
e.g.~\cite{Se, B}). Weights in $\ov{\h}^*$ corresponding to the same
central character can be characterized as follows.  Let $P$ be the
free abelian group with basis $\{\varepsilon_a\vert a\in\Z\}$.
Define
\begin{equation*}
\varepsilon_\gamma:=\sum_{i=-m}^{-1}\varepsilon_{\gamma_i-i}-
\sum_{j=1}^{n}\varepsilon_{-\gamma_j+j}.
\end{equation*}
Then $\chi_\gamma=\chi_\mu$ if and only if
$\varepsilon_\gamma=\varepsilon_\mu$  (see \cite[Lemma 4.18]{B}).
Returning to tensor modules, take a partition
$\la=(\la_{-m},\cdots,\la_{-1},\la_1,\la_2,\cdots)$ with $\la_1\le
n$. Consider
\begin{equation*}
\varepsilon_{\la^\natural}=\sum_{i=-m}^{-1}\varepsilon_{\la_i-i}-
\sum_{j=1}^{n}\varepsilon_{-\la'_j+j}.
\end{equation*}

\begin{lem}\label{auxlemma} Let $\la\in\mc{P}^+_{m|n}$ and
suppose that for some $i,j>0$ we have $\la_{-j}+j=-\la'_i+i$.  Then
we have $\la_{-j}=i-j$, $\la'_i=0$, and hence $\la'_s=0$, for
$s=i+1,\cdots,n$.
\end{lem}

\begin{proof}
The equality $\la_{-j}+j=-\la'_i+i$ implies that
$\la_{-j}+\la'_i=i-j$.  Now if $\la'_i>0$, then $\la_{-j}\ge i$. But
then the identity implies that $\la_i'=0$.
\end{proof}

\begin{lem} The map given by
$\la\mapsto\varepsilon_{\la^\natural}$ is an injection from
$\mathcal{P}^+_{m|n}$ to $P$.
\end{lem}

\begin{proof}
We will show that we can construct the partition $\la$ from
$\varepsilon_{\la^\natural}$.

In light of the \lemref{auxlemma} we can construct $\la$ as follows.
Consider
\begin{equation}\label{epsilonlambda}
\varepsilon_{\la^\natural}=\varepsilon_{a_1}+\cdots+\varepsilon_{a_s}-\varepsilon_{b_1}-\cdots-\varepsilon_{b_t},
\end{equation}
with $a_1>a_2>\cdots>a_s$, and $b_1<b_2<\cdots<b_t$.  The total
number of terms is $m+n-2k$, where $k$ is the degree of atypicality
of $\la$. Let
\begin{equation}\label{epsilonminus}
\varepsilon^0_{-}=-\varepsilon_{b_1}-\cdots-\varepsilon_{b_t}.
\end{equation}
Assume that  $k>0$. Otherwise it is clear what $\la$ is. Let $c_0\le
n$ be the smallest positive integer such that
$-\varepsilon_{c_0+1}-\varepsilon_{c_0+2}-\cdots-\varepsilon_n$ is a
summand of \eqnref{epsilonminus} above.  Then we define
\begin{equation}\label{epsilonminus1}
\varepsilon^1_{-}=-\varepsilon_{b_1}-\cdots-\varepsilon_{b_t}-\varepsilon_{c_0}.
\end{equation}
Now we let $c_1< c_0$ be the smallest integer such that
$-\varepsilon_{c_1+1}-\cdots-\varepsilon_n$ is a summand of
\eqnref{epsilonminus1}.  Define $\varepsilon^2_-$ analogously et
cetera until we get
\begin{equation*}
\varepsilon_-^k=\varepsilon^0_--\sum_{i=0}^{k-1}\varepsilon_{c_i}.
\end{equation*}
Consider
\begin{equation*}
\varepsilon_{\la^\natural}=\big(\varepsilon_{a_1}+\cdots+\varepsilon_{a_s}+\sum_{i=0}^{k-1}\varepsilon_{c_i}\big)
-\big(\varepsilon_{b_1}+\cdots+\varepsilon_{b_t}+\sum_{i=0}^{k-1}\varepsilon_{c_i}\big).
\end{equation*}
From this form with $m+n$ summands we can read off $\la$ easily.
\end{proof}

\begin{cor}\label{ext:finite} Let $n\in\N$. The irreducible tensor representations have distinct central
characters.  In particular for $\la,\mu\in\mc P^+_{m|n}$ we have
$${\rm
Ext}^1\left(L(\gl(m|n),\la^\natural),L(\gl(m|n),\mu^{\natural})\right)=0.$$
\end{cor}

Let $\mathcal O_{m|n}^{++}$ denote the full subcategory of $\mathcal
O_{m|n}^{+}$ such that each composition factor isomorphic to an
irreducible tensor module.

\begin{rem} Let $n\in\N$ and let
$\mc{O}^{++,f}_{m|n}$ denote the full subcategory of
$\mc{O}^{++}_{m|n}$ consisting of objects that have finite
composition series. Then $\mc{O}^{++,f}_{m|n}$ is a full subcategory
of the category of finite-dimensional $\gl(m|n)$-modules. From
\corref{ext:finite} and the long exact sequence one shows, using
induction on the length of the composition series, that ${\rm
Ext}^1(M,N)=0$, for $M,N\in \mc{O}^{++,f}_{m|n}$. The projective
cover $P(\gl(m|n),\la^\natural)$ of $L(\gl(m|n),\la^\natural)$
exists in the category of finite-dimensional $\gl(m|n)$-modules (see
e.g.~\cite{B}). Now consider the short exact sequence
\begin{equation*}
0\longrightarrow K\longrightarrow
P(\gl(m|n),\la^\natural)\longrightarrow
L(\gl(m|n),\la^\natural)\longrightarrow 0.
\end{equation*}
From the long exact sequence one shows that ${\rm Ext}^i(M,N)=0$,
and in particular, ${\rm
Ext}^i\left(L(\gl(m|n),\la^\natural),L(\gl(m|n),\mu^{\natural})\right)=0$,
for all $i\in\N$.
\end{rem}

\begin{thm}\label{cat:ss}
For $n\in\N\cup\{\infty\}$ the category $\mathcal O_{m|n}^{++}$ is a
semisimple tensor category.
\end{thm}

\begin{proof}
Given $M,N\in\mc{O}^{++}_{m|n}$ with composition series $0\subseteq
M_1\subseteq M_2\subseteq\ldots$, and $0\subseteq N_1\subseteq
N_2\subseteq\ldots$, respectively. We obtain a filtration of
$M\otimes N$
\begin{align*}
0\subseteq M_1\otimes N_1\subseteq M_2\otimes N_2\subseteq\ldots,
\end{align*}
whose composition series is a composition series for $M\otimes N$.
Now for $\la,\mu,\nu\in \mc{P}^+_{m|n}$, $L(\la^\natural)\subseteq
L(\mu^\natural)\otimes L(\nu^\natural)$ only if $|\la|=|\mu|+|\nu|$.
Thus every isotypic composition factor of $M\otimes N$ appears with
finite multiplicity.  Hence $\mc{O}^{++}_{m|n}$ is a tensor
category.

The theorem for a finite $n$ now follows from \corref{ext:finite}.
To show the statement for $n=\infty$ we need to show that any exact
sequence of $\gl(m|\infty)$-modules of the form
\begin{equation*}
0\longrightarrow L(\gl(m|\infty),\la^\natural)\longrightarrow
E\longrightarrow L(\gl(m|\infty),\mu^\natural)\longrightarrow 0
\end{equation*}
is split.  For this choose a finite $n$ sufficiently large so that
$\la_{1},\mu_{1}\le n$.  Applying the truncation functor $\mf{tr}_n$
we get an exact sequence of $\gl(m|n)$-modules
\begin{equation*}
0\longrightarrow L(\gl(m|n),\la^\natural)\longrightarrow \mf{tr}_n
E\longrightarrow L(\gl(m|n),\mu^\natural)\longrightarrow 0,
\end{equation*}
which is split. Now repeating the arguments of
\lemref{finite:to:infinite} completes the proof.
\end{proof}

\subsubsection{Non-standard Borel subalgebra of
$\gl(m|m)$}\label{nonstandardborel}

 Below we will discuss the highest weight theory of the
Lie superalgebra $\gl(m|m)$, for $m\in\N\cup\{\infty\}$, with
respect to the non-standard Borel subalgebra.

Let $\C^{m|m}=\C^{>0}$ be the superspace of dimension $(m|m)$. We
assume that it has a basis $\{\,e_r\,|\,
r=\hf,1,\frac{3}{2},2,\cdots\}$ with ${\rm deg}e_r=\bar{0}$ for
$r\in\{1,2,\cdots\}$, and ${\rm deg}e_s=\bar{1}$  otherwise. With
respect to this basis the subalgebra generated by $E_{rs}$ ($r\le
s$) will be called the {\em non-standard Borel subalgebra} of
$\gl(m|m)$.

For a finite $m$ it is clear that $(\C^{m|m})^{\otimes k}$ is a
completely reducible $\gl(m|m)$-module, and the irreducible factors
are the irreducible tensor representations, which are parameterized
by $\la=(\la_1,\la_2,\ldots)\in\mc{P}^+$ with $\la_{m+1}\le m$. In
fact, for such a $\la$, the highest weight (with respect to the
non-standard Borel subalgebra and the Cartan subalgebra $\sum_{r}\C
E_{rr}$) of the corresponding representation of $\gl(m|m)$ is given
by $\sum_{r}a_r\delta_r$ with (cf.~e.g.~\cite[Section 7]{CW2})
\begin{align*}
& a_r=
\begin{cases}
\langle \la'_{r+\hf} -(r-\hf)\rangle, & \text{if $r\in\{\hf,\frac{3}{2},\cdots\}$}, \\
\langle \la_{r} -r\rangle, & \text{if $r\in\{1,2,\cdots\}$}.
\end{cases}
\end{align*}
Above $\langle k\rangle=k$, for $k>0$, and $\langle k\rangle=0$,
otherwise. By \eqnref{hookschur} its character, which is the trace
of the operator $\prod_{r}x_r^{E_{rr}}$, is
$HS_{\la'}(x_\hf,x_1,x_{\frac{3}{2}},\cdots)$, since the standard
and non-standard Borel subalgebras share the same Cartan subalgebra.

For $m=\infty$ it is not hard to show, with the help of a similarly
defined truncation functor, that an analogous statement of
\lemref{finite:to:infinite} remains valid in the setup of
non-standard Borel.

We define respective categories $\mc{O}^{++}_{m|m}$
($m\in\N\cup\{\infty\}$) which is a semisimple tensor category for
$m\in\N$ by \thmref{cat:ss}. Using this fact we then can show in an
analogous fashion the following.

\begin{thm}\label{catnon:ss} The category
$\mathcal O_{\infty|\infty}^{++}$ is a semisimple tensor category.
\end{thm}

To simplify notation for later on we set, for $m=\infty$ above,
$\ov{\gl}_{>0}:=\gl(\infty|\infty)$, which may be regarded as a
subalgebra of $\hgl_{\infty|\infty}$ spanned by $E_{rs}$ with $r,s
>0$. Let $\overline{\gl}_{\le 0}$ be the subalgebra of
$\hgl_{\dinfty}$ spanned by $E_{rs}$ with $r,s\leq 0$, which is
isomorphic to $\overline{\gl}_{> 0}$.

Let $\C^{\le 0}=\sum_{r\in -\hf\Z_+}\C e_r$. Similarly, we call an
irreducible component of $(\C^{\le 0 *})^{\otimes k}$ an {\em
irreducible dual tensor representation} of $\overline{\gl}_{\le 0}$.
As in $(\C^{>0})^{\otimes k}$, each irreducible component
corresponds to $\la\in\mc{P}^+$, whose highest weight with respect
to the non-standard Borel subalgebra $\sum_{r\le s}\C E_{rs}$ and
the Cartan subalgebra $\sum_{r \in-\hf\Z_+}\C E_{rr}$ is given by
\begin{equation*}
-\sum_{r \in-\hf\Z_+}b_r\delta_r,\quad\text{where }b_r=
\begin{cases}
\langle \la_{-r+1}+r\rangle, & \text{if $r\in -\Z_+$}, \\
\langle \la'_{-r+\hf} +(r-\hf)\rangle, & \text{if $r\in -\hf-\Z_+$}.
\end{cases}
\end{equation*}
Its character is
$HS_{\la}(x_0^{-1},x_{-\hf}^{-1},x_{-1}^{-1},\ldots)$ and we have an
analogous statement on semisimplicity of an analogous category of
dual tensor representations.

\subsection{Super dual pairs on infinite-dimensional Fock spaces}\label{super:dualpairs}

Let $\SG$ be one of the Lie superalgebras defined in
\secref{ex:superalgebras}. Let $\La\in\bar{\h}^*$ be given. By
standard arguments there exists  a unique irreducible highest weight
$\SG$-module of highest weight $\La$, which will be denoted by
${L}(\SG,\La)$. For $\La\in\bar{\h}^*$, we set
$\La_s=\langle\La,E_{ss}\rangle$ ($s\in\hf\Z$) if
$\SG=\hgl_{\dinfty}$, and
$\La_s=\langle\La,\widetilde{E}_{s}\rangle$ ($s\in\hf\N$) if
$\SG=\widehat{\mc B},\widehat{\mc C},\widehat{\mc D}$.

Fix a positive integer $\ell\geq 1$. Consider $\ell$ pairs of free
fermions $\psi^{\pm,i}(z)$ and $\ell$ pairs of free bosons
$\gamma^{\pm,i}(z)$ with $i=1,\cdots,\ell$. That is, we have
\begin{align*}
\psi^{+,i}(z)&=\sum_{n\in\Z}\psi^{+,i}_nz^{-n-1},\quad\quad\
\psi^{-,i}(z)=\sum_{n\in\Z}\psi^{-,i}_nz^{-n}, \\
\gamma^{+,i}(z)&=\sum_{r\in\frac{1}{2}+\Z}\gamma^{+,i}_rz^{-r-1/2},\quad
\gamma^{-,i}(z)=\sum_{r\in\frac{1}{2}+\Z}\gamma^{-,i}_rz^{-r-1/2},
\end{align*}
with non-trivial commutation relations
$$[\psi^{+,i}_m,\psi^{-,j}_n]=\delta_{ij}\delta_{m+n,0},\quad
[\gamma^{+,i}_r,\gamma^{-,j}_s]=\delta_{ij}\delta_{r+s,0}.$$ Let
$\SF^\ell$ denote the corresponding Fock space generated by the
vaccum vector $|0\rangle$, which is annihilated by
$\psi^{+,i}_n,\psi^{-,i}_m,\gamma^{\pm,i}_r$ for $n\ge 0$, $m>0$ and
$r>0$.

We also consider the following $\ell$ pairs of free fermions
$\widetilde{\psi}^{\pm,i}$ ($i=1,\cdots,\ell$):
\begin{align*}
\widetilde{\psi}^{+,i}(z)&=\sum_{n\in\Z\setminus\{0\}}\psi^{+,i}_nz^{-n-1},\quad\quad\
\widetilde{\psi}^{-,i}(z)=\sum_{n\in\Z\setminus\{0\}}\psi^{-,i}_nz^{-n},
\end{align*}
with the same commutation relations. We denote by $\SF_0^\ell$ the
Fock space corresponding to $\widetilde{\psi}^{\pm,i}(z)$ and
$\gamma^{\pm,i}(z)$ ($i=1,\cdots,\ell$) generated by the vacuum
vector $|0\rangle$ with
$\psi^{\pm,i}_n|0\rangle=\gamma^{\pm,i}_r|0\rangle=0$, for $n,r>0$.

We introduce a neutral fermionic field
$\widetilde{\phi}(z)=\sum_{n\in\Z\setminus\{0\}}\phi_nz^{-n-1}$ and
a neutral bosonic field $\chi(z)=\sum_{r\in\hf+\Z}\chi_rz^{-r-\hf}$
with non-trivial commutation relations
$[\phi_m,\phi_n]=\delta_{m+n,0}$ and
$[\chi_r,\chi_s]=\delta_{r+s,0}$. Denote by $\SF_0^{\hf}$ the
corresponding Fock space generated by the vacuum vector $|0\rangle$,
which is annihilated by $\phi_m$ and $\chi_r$ for $m,r>0$. We denote
by $\SF_0^{\ell+\hf}$ the tensor product of $\SF_0^{\ell}$ and
$\SF_0^{\hf}$.

Let $W$ be the Weyl group of the Lie algebra $\mf{g}$, and $G$ the
dual partner of $\G$ in \secref{classical:dualpairs}. We let
$W_k^0$, $\lambda_w^{\pm}$, $\lambda_w$ ($\la\in \mc{P}(G)$, $w\in
W_k^0$) be as in \secref{classical:homology}. Let $x_r$
($r\in\hf\Z$) and $z_i$ ($i=1,\cdots,\ell$) be formal
indeterminates.

\subsubsection{The $(\hgl_\dinfty,{\rm GL}(\ell))$-duality}
Given $\la\in {\mc P}({\rm GL}(\ell))$, we define $\SLa^{\mf
a}(\la)\in\bar{\h}^*$ by {\allowdisplaybreaks
\begin{align*}
 &\SLa^{\mf a}(\la)_i= \langle\la_i'-i \rangle,\quad i\in\N,\\
 &\SLa^{\mf a}(\la)_j= -\langle-\la'_j+j \rangle, \quad j\in -\Z_{+},\\
 &\SLa^{\mf a}(\la)_r=\langle\la_{r+1/2}-(r-1/2) \rangle,\quad r\in\hf+\Z_+,\\
 &\SLa^{\mf a}(\la)_{s}=-\langle -\la_{\ell+(s+1/2)}+({s-1/2}) \rangle,\quad
 s\in-\hf-\Z_+,\\
&\langle \SLa^{\mf a}(\la),K\rangle=\ell.
\end{align*}}

\begin{prop} \cite[Theorem 8.1]{CW2}\label{sduality-a} There exists an action of $\hgl_\dinfty\times{\rm
GL}(\ell)$ on $\SF^\ell$. Furthermore, under this joint action, we
have
\begin{equation}\label{aux:sdecomp-a}
\SF^\ell\cong\bigoplus_{\la\in  {\mc P}({\rm GL}(\ell))}
L(\hgl_{\infty|\infty},\SLa^{\mf a}(\la))\otimes V_{{\rm
GL}(\ell)}^{\la}.
\end{equation}
\end{prop}

 Computing the trace of the operator
$\prod_{s\in\hf\Z}x_s^{E_{ss}}\prod_{i=1}^\ell z_i^{e_{ii}}$ on both
sides of \eqnref{aux:sdecomp-a}, we obtain the following identity:

\begin{equation}\label{character1}
\prod_{i=1}^\ell\frac{\prod_{n\in\N}(1+x_n z_i)
(1+x_{-n+1}^{-1}z_i^{-1})}{\prod_{r\in1/2+\Z_+}
(1-x_rz_i)(1-x_{-r}^{-1}z_i^{-1})}=\sum_{\la\in\mathcal{P}({\rm
GL}(\ell))}{\rm ch}L(\hgl_{\infty\vert\infty},\SLa^{\mf a}(\la)){\rm
ch}V^\la_{{\rm GL}(\ell)}.
\end{equation}
Set ${\bf x}_{>0}=\{\,x_\hf,x_1,x_{\frac{3}{2}},\cdots\,\}$, ${\bf
x}^{-1}_{\le 0}=\{\,x^{-1}_0,x^{-1}_{-\hf},x^{-1}_{-1},\cdots\,\}$,
and
$$\ov{D}^\mf{a}=\prod_{i,j\in\N}\frac{(1-x_{-i+1}^{-1}x_{j})(1-x^{-1}_{-i+\hf}x_{j-\hf})}
{(1+x_{-i+1}^{-1}x_{j-\hf})(1+x^{-1}_{-i+\hf}x_j)}.$$
\begin{thm}\label{char:super-A}
For $\la\in\mathcal{P}({\rm GL}(\ell))$, we have
\begin{align*}
&{\rm ch}L(\hgl_\dinfty,\SLa^{\mf
a}(\la))=\frac{1}{\ov{D}^{\mf{a}}}\sum_{k=0}^\infty\sum_{w\in
W^0_k}(-1)^k HS_{(\la^+_w)'}({\bf x}_{>0}) HS_{\la^-_w}({\bf
x}^{-1}_{\le 0}).
\end{align*}
\end{thm}

\begin{proof} The proof is similar to that of Theorem 3.2 in
\cite{CL}.  First we use \eqnref{homology:schur} and
\eqnref{char:schur} to write \eqnref{combid-classical1} as an
identity in $\{x_{-j}^{-1}\}_{j\in \Z_+}$ and $\{x_{j}\}_{j\in \N}$.
Next we replace $\{x_{-j}^{-1}\}_{j\in \Z_+}$ (resp.
$\{x_{j}\}_{j\in \N}$) with $\{x^{-1}_{-j}\}_{j\in \hf\Z_+}$ (resp.
$\{x_{j}\}_{j\in \hf \N}$) in this new identity, which we regard
first as a symmetric function identity in the variables
$\{x^{-1}_{-j}\}_{j\in \hf+\Z_+}$.  This allows us to apply to it
the involution of symmetric functions that interchanges the
elementary symmetric functions with the complete symmetric functions
(see e.g.~\cite[Chapter I, \S2, (2.7)]{M}) and obtain, using
\lemref{aux:hookSchur},
\begin{align}\label{combid-super-a}
&\prod_{k=1}^\ell\prod_{i\in\N}\prod_{j\in\hf\N}
\frac{(1+x_jz_k)(1+x^{-1}_{-i+1}z_k^{-1})}{(1-x^{-1}_{-i+\hf}z_k^{-1})}=\sum_{\la\in\mathcal{P}({\rm
GL}(\ell))}{\rm ch}V^\la_{{\rm
GL}(\ell)}\times\nonumber\\
&\frac{\sum_{k=0}^\infty\sum_{w\in
W^0_k}(-1)^ks_{\la^+_w}(x_\hf,x_1,\cdots)
HS_{\la^-_w}(x^{-1}_0,x^{-1}_{-\hf},x^{-1}_{-1},\cdots)}
{\prod_{i\in\N,j\in\hf\N}(1-x^{-1}_{-i+1}x_j)/(1+x^{-1}_{-i+\hf}x_j)}.
\end{align}
Now we regard \eqnref{combid-super-a} as a symmetric function
identity in the variables $\{x_j\}_{j\in\hf+\Z_+}$ and apply the
involution again to obtain, using \lemref{aux:hookSchur},
\begin{align}
\prod_{k=1}^\ell&\frac{\prod_{n\in\N}(1+x_nz_k)
(1+x_{-n+1}^{-1}z_k^{-1})}{\prod_{r\in1/2+\Z_+}
(1-x_rz_k)(1-x_{-r}^{-1}z_k^{-1})}
=\sum_{\la\in\mathcal{P}({\rm GL}(\ell))}{\rm ch}V^\la_{{\rm GL}(\ell)}\times \nonumber\\
&\frac{\sum_{k=0}^\infty\sum_{w\in W^0_k}(-1)^k
HS_{(\la^+_w)'}(x_\hf,x_1,\cdots)
HS_{\la^-_w}(x^{-1}_0,x^{-1}_{-\hf},x^{-1}_{-1},\cdots)}
{{\prod_{i,j\in\N}(1-x_{-i+1}^{-1}x_{j})(1-x^{-1}_{-i+\hf}x_{j-\hf})}/
{(1+x_{-i+1}^{-1}x_{j-\hf})(1+x^{-1}_{-i+\hf}x_j)}}.\nonumber
\end{align}
Now the theorem follows from \eqnref{character1} and the fact that
the set  $\{{\rm ch}V^\la_{{\rm
GL}(\ell)}\vert\la\in\mathcal{P}({\rm GL}(\ell))\}$, being the set
of Schur Laurent polynomials, is linearly independent.
\end{proof}

We record an identity (cf.~\cite[Remark 3.3]{CW1}) that can be
proved similarly as \thmref{char:super-A} using the classical Cauchy
identity:
\begin{equation}\label{scauchy}
\prod_{i=1}^\ell\frac{\prod_{n\in\N}(1+x_n z_i)
}{\prod_{r\in1/2+\Z_+}
(1-x_rz_i)}=\sum_{\la\in\mathcal{P}^+,l(\la)\le\ell}{
HS}_{\la'}({\bf x}_{>0}) s_\la(z_1,\cdots,z_\ell).
\end{equation}

\subsubsection{The $(\widehat{\mc{B}},{\rm Pin}(2\ell))$-duality}

For $\lambda\in {\mc P}({\rm Pin}(2\ell))$, we define $\SLa^{\mf
b}(\lambda)\in\bar{\h}^*$ by
\begin{align*}
&\SLa^{\mf b}(\la)_i= \langle\la_i'-i \rangle,\quad i\in\N,\\
&\SLa^{\mf b}(\la)_r=\langle\la_{r+1/2}-(r-1/2) \rangle,\quad r\in\hf+\Z_+,\\
&\langle\SLa^{\mf b}(\la),K\rangle=\ell.
\end{align*}
\begin{prop} \cite[Theorem 8.2]{CW2}\label{sduality for b} There exists an action of $\widehat{\mc{B}}\times{\rm Pin}(2\ell)$ on
$\SF^{\ell}$. Furthermore, under this joint action, we have
\begin{equation}\label{aux:sdecomp-b}
\SF^{\ell}\cong\bigoplus_{\la\in {\mc P}({\rm Pin}(2\ell))}
L(\widehat{\mc{B}},\SLa^{\mf b}(\la))\otimes V_{{\rm
Pin}(2\ell)}^{\la}.
\end{equation}
\end{prop}

Computing the trace of the operator $\prod_{s\in
\hf\N}x_s^{\widetilde{E}_{s}}\prod_{i=1}^\ell z_i^{\tilde{e}_{i}}$
on both sides of \eqnref{aux:sdecomp-b}, we obtain the following
identity:
\begin{equation}\label{combid-super-b}
\prod_{i=1}^\ell(z_i^{\hf}+z_i^{-\hf})
\frac{\prod_{n\in\N}(1+x_{n}z_i)(1+x_{n}z^{-1}_{i})}{\prod_{r\in
\hf+\Z_+}(1-x_rz_i)(1-x_{r}z^{-1}_{i})} =\sum_{\la\in {\mc P}({\rm
Pin}(2\ell))} {\rm ch}L(\widehat{\mc{B}},\overline{\Lambda}^{\mf
b}(\la)){\rm ch}V^\la_{{\rm Pin}(2\ell)}.
\end{equation}

\begin{thm}\label{char:super-B}
For $\la\in\mathcal{P}({\rm Pin}(2\ell))$, we have
\begin{align*}
&{\rm ch}L(\widehat{\mc B},\SLa^{\mf b}(\la))
=\frac{1}{\ov{D}^\mf{b}}{\sum_{k=0}^\infty\sum_{w\in W^0_k}(-1)^k
HS_{(\la_w)'}({\bf x}_{>0})},
\end{align*}
where
$$\ov{D}^\mf{b}=
\prod_{\substack{i,j\in \N \\ i\leq
j}}\frac{(1-x_i)(1-x_{i}x_{j+1})(1-x_{i-\hf}x_{j-\hf})}{(1+x_{i-\hf})(1+x_{i}x_{j+\hf})(1+x_{i-\hf}x_{j})}.$$
\end{thm}
\begin{proof}
The proof is similar to that of \thmref{char:super-A}. Replace
$\{x_{j}\}_{j\in \N}$ with $\{x_{j}\}_{j\in \hf \N}$ in
\eqnref{combid-classical-b}, and then apply the involution to the
symmetric functions in $\{x_{j}\}_{j\in \hf+ \Z_+}$
(cf.~\cite[Chapter I \S 5 Ex. 4,5]{M}). Then the result follows from
the linear independence of $\{\,{\rm ch}V^\la_{{\rm
Pin}(2\ell)}\,|\,\la\in\mc{P}({\rm Pin}(2\ell))\,\}$.
\end{proof}

\subsubsection{The $(\widehat{\mc{C}},{\rm Sp}(2\ell))$-duality} For $\lambda\in
{\mc P}({\rm Sp}(2\ell))$, we define $\SLa^{\mf
c}(\lambda)\in\bar{\h}^*$ by
\begin{align*}
&\SLa^{\mf c}(\la)_i= \langle\la_i'-i \rangle,\quad i\in\N,\\
&\SLa^{\mf c}(\la)_r=\langle\la_{r+1/2}-(r-1/2) \rangle,\quad r\in\hf+\Z_+,\\
&\langle\SLa^{\mf c}(\la),K\rangle=\ell.
\end{align*}

\begin{prop} \cite[Theorem 5.3]{LZ}\label{sduality for c} There exists an action of $\widehat{\mc{C}}\times{\rm
Sp}(2\ell)$ on $\SF_0^{\ell}$.  Furthermore, under this joint
action, we have
\begin{equation}\label{aux:sdecomp-c}
\SF_0^\ell\cong\bigoplus_{\la\in {\mc P}({\rm Sp}(2\ell))}
L(\widehat{\mc{C}},\SLa^{\mf c}(\la))\otimes V_{{\rm
Sp}(2\ell)}^{\la}.
\end{equation}
\end{prop}

Computing the trace of the operator
$\prod_{s\in\hf\N}x_s^{\widetilde{E}_{s}}\prod_{i=1}^\ell
z_i^{\widetilde{e}_{i}}$ on both sides of \eqnref{aux:sdecomp-c}, we
obtain the following:

\begin{equation}\label{combid-super-c}
\prod_{i=1}^\ell\frac{\prod_{n\in\N}(1+x_n z_i)
(1+x_{n}z_i^{-1})}{\prod_{r\in\hf+\Z_+}
(1-x_rz_i)(1-x_{r}z_i^{-1})}=\sum_{\la\in\mathcal{P}({\rm
Sp}(2\ell))}{\rm ch}L(\widehat{\mc C},\SLa^{\mf c}(\la)){\rm
ch}V^\la_{{\rm Sp}(2\ell)}.
\end{equation}

Similar to \thmref{char:super-B}, we obtain the following.
\begin{thm}\label{char:super-C}
For $\la\in\mathcal{P}({\rm Sp}(2\ell))$, we have
\begin{align*}
&{\rm ch}L(\widehat{\mc C},\SLa^{\mf c}(\la)) =\frac{1}{\ov
D^\mf{c}}{\sum_{k=0}^\infty\sum_{w\in W^0_k}(-1)^k
HS_{(\la_w)'}({\bf x}_{>0})},
\end{align*}
where
$${\ov D}^\mf{c}={\prod_{\substack{i,j\in \N \\
i\leq j
}}\frac{(1-x_{i}x_{j})(1-x_{i-\hf}x_{j+\hf})}{(1+x_{i}x_{j+\hf})(1+x_{i-\hf}x_{j})}}.$$
\end{thm}

\subsubsection{The $(\widehat{\mc{D}},{\rm
O}(m))$-duality}\label{omegaappears}

For $\lambda\in {\mc P}({\rm O}(m))$, we define $\SLa^{\mf
d}(\lambda)\in\bar{\h}^*$ by
\begin{align*}
&\SLa^{\mf d}(\la)_i= \langle\la_i'-i \rangle,\quad i\in\N,\\
&\SLa^{\mf d}(\la)_r=\langle\la_{r+1/2}-(r-1/2) \rangle,\quad r\in\hf+\Z_+,\\
&\langle\SLa^{\mf d}(\la),K\rangle=\frac{m}{2}.
\end{align*}
\begin{prop} \cite[Theorems 5.3 and 5.4]{LZ}\label{sduality for d} There exists an action of $\widehat{\mc{D}}\times{\rm
O}(m)$ on $\SF_0^{\frac{m}{2}}$. Furthermore, under this joint
action, we have
\begin{equation}\label{aux:sdecomp-d}
\SF_0^{\frac{m}{2}}\cong\bigoplus_{\la\in {\mc P}({\rm O}(m))}
L(\widehat{\mc{D}},\SLa^{\mf d}(\la))\otimes V_{{\rm O}(m)}^{\la}.
\end{equation}
\end{prop}

Suppose that $m=2\ell$. Computing the trace of the operator
$\prod_{s\in \hf\N}x_s^{\widetilde{E}_{s}}\prod_{i=1}^\ell
z_i^{\tilde{e}_{i}}$ on both sides of \eqnref{aux:sdecomp-d}, we
obtain the following identity:
\begin{equation}\label{combid-super-d1}
\prod_{i=1}^\ell
\frac{\prod_{n\in\N}(1+x_nz_i)(1+x_{n}z^{-1}_{i})}{\prod_{r\in
\hf+\Z_+}(1-x_rz_i)(1-x_{r}z^{-1}_{i})} =\sum_{\la\in {\mc P}({\rm
O}(2\ell))} {\rm ch}L(\widehat{\mc{D}},\overline{\Lambda}^{\mf
d}(\la)){\rm ch}V^\la_{{\rm O}(2\ell)}.
\end{equation}

Suppose that $m=2\ell+1$. Let $\epsilon$ be the eigenvalue of $-I_m$
on ${\rm O}(2\ell+1)$-modules satisfying $\epsilon^2=1$. Calculating
the trace of the operator $\prod_{s\in
\hf\N}x_s^{\widetilde{E}_{s}}\prod_{i=1}^\ell
z_i^{\tilde{e}_{i}}(-I_m)$ on \eqnref{aux:sdecomp-d} yields
\begin{equation}\label{combid-super-d2}
\prod_{i=1}^\ell \frac{\prod_{n\in\N}(1+\epsilon x_nz_i)(1+\epsilon
x_{n}z^{-1}_{i})(1+\epsilon x_n)}{\prod_{r\in \hf+\Z_+}(1-\epsilon
x_rz_i)(1-\epsilon x_{r}z^{-1}_{i})(1-\epsilon x_r)} =\sum_{\la\in
{\mc P}({\rm O}(2\ell+1))} {\rm
ch}L(\widehat{\mc{D}},\overline{\Lambda}^{\mf d}(\la)){\rm
ch}V^\la_{{\rm O}(2\ell+1)}.
\end{equation}

\begin{thm}\label{char:super-D}
Suppose that $\la\in\mathcal{P}({\rm O}(m))$ is given. Put
$${\ov D^\mf{d}}={\prod_{\substack{i,j\in \N \\
i\leq j
}}\frac{(1-x_{i}x_{j+1})(1-x_{i-\hf}x_{j-\hf})}{(1+x_{i}x_{j+\hf})(1+x_{i-\hf}x_{j})}}.$$
\begin{itemize}
\item[(1)] If $m=2\ell+1$, then we have
\begin{align*}
&{\rm ch}L(\widehat{\mc D},\SLa^{\mf d}(\la))
=\frac{1}{\ov{D}^\mf{d}}{\sum_{k=0}^\infty\sum_{w\in W^0_k}(-1)^k
HS_{(\la_w)'}({\bf x}_{>0})} .
\end{align*}

\item[(2)] If $m=2\ell$, then we have
\begin{align*}
&{\rm ch}L(\widehat{\mc D},\SLa^{\mf d}(\la))
+{\rm ch}L(\widehat{\mc D},\SLa^{\mf d}(\tilde{\la})) \\
& =\frac{1}{\ov{D}^\mf{d}}{\sum_{k=0}^\infty\sum_{w\in W^0_k}(-1)^k
\left[ HS_{(\la_w)'}({\bf x}_{>0})+HS_{(\tilde{\la}_w)'}({\bf
x}_{>0}) \right]}.
\end{align*}
\end{itemize}
\end{thm}
\begin{proof}
The proof is similar to that of \thmref{char:super-B}. However, note
that in the case of (2), $V^\la_{{\rm O}(2\ell)}$ and
$V^{\tilde{\la}}_{{\rm O}(2\ell)}$ are isomorphic as
$\mf{so}(2\ell)$-modules, and they give the same character. Hence,
the result follows from comparing the coefficients of ${\rm
ch}V^{{\la}}_{{\rm O}(2\ell)}$ for $\la\in\mc{P}({\rm O}(2\ell))$
with length less than or equal to $\ell$ (cf. ~\cite[Lemma
6.1]{CZ2}).
\end{proof}

From now on we mean by $(\overline{\mf{g}},G)$ one of the dual pairs
of \secref{super:dualpairs}, and by $\mf{x}\in\{\mf{a,b,c,d}\}$ the
type of $\SG$. The map (consisting of involutions of symmetric
functions) that transforms (\ref{char:schur}) into ${\rm
ch}L(\SG,\La^\mf{x}(\la))$ of the form in Theorems
\ref{char:super-A}-\ref{char:super-D} will be denoted by
$\omega^\mf{x}$.

\section{The Casimir operators}\label{casimir:op}

\subsection{The bilinear form $(\cdot|\cdot)_c$ and the Casimir operator $\Omega$ of
$\G$}\label{biformc}

 Suppose first that $\G=\hgl_{\infty}$.  We fix a symmetric
bilinear form $(\cdot\vert\cdot)_c$ on $\h^*$ satisfying
\begin{align*}
&(\la\vert \epsilon_i)_c=\langle \la,E_{ii}-{\zeta}(i)\frac{K}{2}\rangle,  \quad \la\in\h^*,i\in\Z,\\
&(\La^{\mf a}_0\vert\La^{\mf a}_0)_c=(\La^{\mf
a}_0\vert\rho_{c})_c=0,
\end{align*}
where ${\zeta}(i)=1$ (resp. $-1$) if $i> 0$ (resp. $i\leq 0$). Such
a form exists, since the vectors $\{\epsilon_i\}_{i\in\Z}$,
$\La^{\mf a}_0$ and $\rho_{c}$ are linearly independent. We check
easily that
\begin{align*}
&(\epsilon_i\vert\epsilon_j)_c=\delta_{ij},\quad (\Lambda^{\mf
a}_0\vert\epsilon_i)_c=-\frac{\zeta(i)}{2},\quad i,j\in \Z,\\
&(\rho_c\vert\alpha_i)_c=\frac{1}{2}(\alpha_i\vert\alpha_i)_c, \quad
i\in I.
\end{align*}

Suppose that $\G=\mf{x}_{\infty}$ with $\mf{x}\in\{\mf{b,c,d}\}$. We
choose a symmetric bilinear form $(\cdot\vert\cdot)_c$ on $\h^*$
satisfying
\begin{align*}
&(\la\vert \epsilon_i)_c =\langle \la,\widetilde{E}_{i}-K\rangle,
\quad i\in\N,\\
&(\La^\mf{x}_0\vert\La^\mf{x}_0)_c=(\La^\mf{x}_0\vert\rho_c)_c =0.
\end{align*}
Also one checks that $(\epsilon_i\vert\epsilon_j)_c=\delta_{ij}$,
$(\Lambda^\mf{x}_0\vert\epsilon_i)_c=-r$ for $i,j\in \N$, where
$r=\hf,1,\hf$  for $\mf{x}=\mf{b,c,d}$, respectively, and
$2(\rho_c\vert\alpha_i)_c=(\alpha_i\vert\alpha_i)_c$ for $i\in I$.

Now for each $\G$ of type $\mf{x}\in\{\mf{a,b,c,d}\}$, let
$\{\,s^\mf{x}_i\,\}_{i\in I}$ be the sequence defined by
\begin{equation*}
s^{\mf a}_i=s^{\mf d}_i=1, \ \ i\in I,\quad s^{\mf b}_i=
\begin{cases}
\hf, & \text{if $i=0$}, \\
1, & \text{if $i\geq 1$},
\end{cases} \quad
s^{\mf c}_i=
\begin{cases}
2, & \text{if $i=0$}, \\
1, & \text{if $i\geq 1$}.
\end{cases}
\end{equation*}
Then it follows that
\begin{equation}\label{aux:casimir1}
(\la\vert\alpha_i)_c=s^{\mf{x}}_i\langle\la,\alpha^{\vee}_i\rangle,
\end{equation}
for $\la\in{\h}^*$, $i\in I$ so that
$(\alpha_i\vert\alpha_j)_c=s^\mf{x}_j\langle\alpha_i,{\alpha}^{\vee}_j\rangle$,
for $i,j\in I$. By defining $({\alpha}^{\vee}_i\vert
{\alpha}^{\vee}_j)_c:=(s_i^\mf{x}s_j^\mf{x})^{-1}(\alpha_i\vert\alpha_j)_c$,
we obtain a symmetric bilinear form on the Cartan subalgebra of
$\G'=[\G,\G]$, which can be extended to a non-degenerate invariant
symmetric bilinear form on $\G'$ such that
\begin{equation}\label{aux:casimir2}
(e_i\vert f_j)_c=\delta_{ij}/s^{\mf{x}}_i,
\end{equation}
where $e_i$ and $f_j$ $(i,j\in I)$ denote the Chevalley generators
of $\G'$ with $[e_i,f_i]=\alpha_i^{\vee}$. Thus the symmetric
bilinear form on $\h^*$ induces an invariant symmetric
non-degenerate bilinear form on the derived subalgebra of $\G$.
Since every root space is one-dimensional, we can choose a basis
$\{u_{\alpha}\}$ of $\G_\alpha$ for $\alpha\in\Delta^+$ and a dual
basis $\{u^{\alpha}\}$ of $\G_{-\alpha}$ with respect to
$(\cdot\vert\cdot)_c$.

Let $V$ be a highest weight $\G$-module with weight space
decomposition $V=\oplus_{\mu} V_\mu$. Define $\Gamma_1:V\rightarrow
V$ to be the linear map that acts as the scalar
$(\mu+2\rho_c\vert\mu)_c$ on $V_\mu$. Let
$\Gamma_2:=2\sum_{\alpha\in\Delta^+}u^{\alpha}u_{\alpha}$. Define
the Casimir operator (cf.~\cite{J}) to be
\begin{equation*}
\Omega:=\Gamma_1+\Gamma_2.
\end{equation*}
It follows from \eqnref{aux:casimir1} and \eqnref{aux:casimir2} that
$\Omega$ commutes with the action of $\G$ on $V$
(cf.~\cite[Proposition 3.6]{J}). Thus, if $V$ is generated by a
highest weight vector with highest weight $\la$, then $\Omega$ acts
on $V$ as the scalar $(\la+2\rho_c\vert\la)_c$.

\subsection{The bilinear form $(\cdot|\cdot)_s$ and the Casimir operator $\overline{\Omega}$ of
$\SG$}\label{rhos:aux}

Suppose first that $\SG=\hgl_{\dinfty}$. Define $\SLa^{\mf
a}_0\in\bar{\h}^*$ by $\langle\SLa^{\mf a}_0,K \rangle=1$ and
$\langle\SLa^{\mf a}_0,E_{rr}\rangle=0$, for all $r\in\hf\Z$. We
choose a symmetric bilinear form $(\cdot\vert\cdot)_s$ on
$\ov{\h}^*$ satisfying
\begin{align*}
&(\la\vert \delta_r)_s=(-1)^{2r}\langle\la,E_{rr}-(-1)^{2r}{\zeta}(r)\frac{K}{2}\rangle, \quad \la\in\ov{\h}^*, r\in\hf\Z,\\
&(\SLa^{\mf a}_0\vert\SLa^{\mf a}_0)_s=0.
\end{align*}
Here we set that $\rho_s=0\in \ov{\h}^*$. We check easily that
\begin{align*}
&(\delta_r\vert\delta_t)_s=(-1)^{2r}\delta_{rt},\quad (\SLa^{\mf
a}_0\vert\delta_r)_s=-\frac{\zeta(r)}{2}, \quad r,t\in\hf\Z,\\
&(\rho_s\vert\beta_r)_s=\frac{1}{2}(\beta_r\vert\beta_r)_s, \quad
r\in \overline{I}.
\end{align*}

Suppose that $\SG=\widehat{\mc{X}}$ for $\mc{X}\in\{\mc{B,C,D}\}$,
with $\mf{x}\in\{\mf{b,c,d}\}$ denoting the respective type of
$\SG$. Define $\SLa^\mf{x}_0\in \overline{\mf{h}}^*$ by
$\langle\SLa^\mf{x}_0,K\rangle=1$ and
$\langle\SLa^\mf{x}_0,\widetilde{E}_r\rangle=0$ for $r\in \hf\N$.
For $\widehat{\mc{C}}$, we put $\rho_s=0$. For the other cases, let
$\rho_s\in\overline{\h}^*$ be determined by
\begin{equation*}
\begin{aligned}
\langle \rho_s,\widetilde{E}_{r}\rangle&=
\begin{cases}
(-1)^{2r}\frac{1}{2}, & \text{for $\widehat{\mc{B}}$}, \\
(-1)^{2r}, & \text{for $\widehat{\mc{D}}$},
\end{cases}\quad r\in\hf\N, \\
\langle \rho_s,K\rangle &=0.
\end{aligned}
\end{equation*}
We fix a symmetric bilinear form $(\cdot\vert\cdot)_s$ on
$\overline{\h}^*$ satisfying
\begin{align*}
&(\la\vert \delta_r)_s=(-1)^{2r}\langle
\la,\widetilde{E}_{r}-(-1)^{2r}K\rangle, \quad \la\in\ov{\h}^*,r\in\hf\N,\\
&(\SLa^\mf{x}_0\vert\SLa^\mf{x}_0)_s=(\SLa^\mf{x}_0\vert\rho_s)_s=0.
\end{align*}
Then we can check that
$(\delta_r\vert\delta_t)_s=(-1)^{2r}\delta_{rt}$,
$(\SLa^\mf{x}_0\vert\delta_r)_s=-1$ for $r,t\in \hf\N$, and
$2(\rho_s\vert\beta_r)_s=(\beta_r\vert\beta_r)_s$ for $r\in
\overline{I}$.

Now, for each $\SG$ of type $\mf{x}\in\{\mf{a,b,c,d}\}$, let
$\{\,\overline{s}^\mf{x}_r\,\}_{r\in \overline{I}}$ be the sequence
defined by
\begin{align*}
&\overline{s}^{\mf a}_r=\overline{s}^{\mf b}_r=\overline{s}^{\mf
c}_r =(-1)^{2r}, \ \ r\in \overline{I},\\ &\overline{s}^{\mf d}_r=
\begin{cases}
2, & \text{if $r=0$}, \\
(-1)^{2r}, & \text{if $r\geq \hf$}.
\end{cases}
\end{align*}
Then we have
\begin{equation}\label{aux:scasimir1}
(\la\vert\beta_r)_s=\overline{s}^\mf{x}_r\langle\la,\beta^{\vee}_r\rangle,
\quad \la\in\ov{\h}^*, r\in\overline{I},
\end{equation}
so that
$(\beta_r\vert\beta_t)_s=\overline{s}^\mf{x}_t\langle\beta_r,{\beta}^{\vee}_t\rangle$
for $r,t\in\overline{I}$. By defining
$({\beta}^{\vee}_r\vert{\beta}^{\vee}_t)_s:=(\ov{s}^\mf{x}_r\ov{s}_t^\mf{x})^{-1}(\beta_r\vert\beta_t)_s$,
we obtain a symmetric bilinear form on the Cartan subalgebra of
$\SG'=[\SG,\SG]$, which can be extended to a non-degenerate
invariant super-symmetric bilinear form on $\SG'$ such that
\begin{equation}\label{aux:scasimir2}
(\ov{e}_r\vert \ov{f}_t)_s=\delta_{rt}/\overline{s}^\mf{x}_r,
\end{equation}
where $\ov{e}_r$ and $\ov{f}_t$ ($r,t\in\overline{I}$) denote the
Chevalley generators of $\SG$ with
$[\ov{e}_r,\ov{f}_r]=\beta^{\vee}_r$.

We have now all the ingredients to define the super-analogue of the
Casimir operator. Namely, for $\beta\in\ov{\Delta}^+$, let
$\SG_\beta$ be the root space of $\SG$ corresponding to $\beta$.
Take a basis $\{u_{\beta}\}$ of $\SG_\beta$, and a dual basis
$\{u^{\beta}\}$ of $\SG_{-\beta}$ with respect to
$(\cdot\vert\cdot)_s$. For any highest weight $\SG$-module $V$, with
weight space decomposition $V=\oplus_{\mu} V_\mu$, we define
$\overline{\Gamma}_1:V\rightarrow V$ to be the linear map that acts
as the scalar $(\mu+2\rho_s\vert\mu)_s$ on $V_\mu$. Let
$\overline{\Gamma}_2:=2\sum_{\beta\in\overline{\Delta}^+}u^{\beta}u_{\beta}$.
For example, if $\SG=\hgl_{\dinfty}$, then we have for $r\in \ov{I}$
\begin{equation*}
u^{\beta_r}u_{\beta_r}=(-1)^{2r}E_{r+\hf,r}E_{r,r+\hf}.
\end{equation*}
Define the Casimir operator to be
\begin{equation}\label{supercasimir}
\overline{\Omega}:=\overline{\Gamma}_1+\overline{\Gamma}_2.
\end{equation}

\begin{prop} Let $V$ be a highest weight $\SG$-module.
The operator $\overline{\Omega}$ commutes with the action of $\SG$
on $V$. In particular, if $V$ is generated by a highest weight
vector with highest weight $\la$, then $\overline{\Omega}$ acts on
$V$ as the scalar $(\la+2\rho_s\vert\la)_s$.
\end{prop}

\begin{proof}
The argument is parallel to the one given to prove \cite[Proposition
3.6]{J}. Here we use \eqnref{aux:scasimir1} and
\eqnref{aux:scasimir2} to replace the corresponding identities of
\cite{J}.
\end{proof}

\subsection{$\ov{\U}_-$-homology groups of $\SG$-modules}

Suppose that $(\G,G)$ and $(\SG,G)$ are the dual pairs of type
$\mf{x}\in\{\mf{a,b,c,d}\}$ given in \secref{classical:dualpairs}
and \secref{super:dualpairs}, and $\la\in\mc{P}(G)$. Recall the
following results for integrable modules of generalized Kac-Moody
algebras \cite{J, L} applied to our setting.

\begin{prop}\label{eigenvalue:classical} \mbox{}
\begin{itemize}
\item[(i)]
If $\eta\in\mathcal P^+_{\mathfrak l}$ is a weight in
$\Lambda^k\mathfrak u_-\otimes L(\G,\La^\mf{x}(\la))$ with
$(\eta+2\rho_c\vert\eta)_c=(\La^\mf{x}(\la)+2\rho_c\vert\La^\mf{x}(\la))_c$,
then there exists $w\in W^0_k$ with $\eta=w\circ\La^\mf{x}(\la)$ and
$\eta$ appears with multiplicity one.
\item[(ii)] The $\mathfrak l$-module ${\rm
H}_k({\mathfrak u}_-;L(\G,\La^\mf{x}(\la)))$ is completely
reducible. Furthermore if $L({\mathfrak l},\eta)$ is an irreducible
component of ${\rm H}_k({\mathfrak u}_-;L(\G,\La^\mf{x}(\la)))$,
then
$(\eta+2\rho_c\vert\eta)_c=(\La^\mf{x}(\la)+2\rho_c\vert\La^\mf{x}(\la))_c$.
\end{itemize}
\end{prop}

Let $\ov{\Delta}:=\ov{\Delta}^+\cup \ov{\Delta}^{\,-}$ be the set of
roots of $\SG$, where $\ov{\Delta}^{\,-}=-\ov{\Delta}^+$. Let
$\ov{\Delta}_S^\pm:=\ov{\Delta}^\pm\cap(\sum_{r\neq 0}\Z\beta_r)$
and $\ov{\Delta}^\pm(S):=\ov{\Delta}^\pm\setminus\ov{\Delta}^\pm_S$.
Let
\begin{equation}\label{superparabolic}
\begin{aligned}
&\ov{\mf u}_{\pm} := \sum_{\beta\in \ov{\Delta}^\pm(S)}\SG_{\beta},
\quad \ov{\mf l}  := \sum_{\beta\in
\ov{\Delta}_S^+\cup\ov{\Delta}_S^{-}}\SG_{\beta}\oplus\ov{\h},\quad\ov{\mf{p}}:=\ov{\mf
l}\oplus\ov{\mf u}_+.
\end{aligned}
\end{equation}
Then we have
$\SG=\ov{\mf{u}}_+\oplus\ov{\mf{l}}\oplus\ov{\mf{u}}_-$. The Lie
superalgebras $\ov{\mf{l}}$ and $\SG$ share the same Cartan
subalgebra. It is not difficult to see that $\ov{\mf{l}}=\ov{\gl}_{>
0}\oplus \ov{\gl}_{\le 0}\oplus\C K$ if $\SG=\hgl_{\dinfty}$, and
$\ov{\mf{l}}\cong \ov{\gl}_{> 0}\oplus\C K$ otherwise.

For $\mu\in\ov{\h}^*$ we denote by ${L}(\ov{\mf l},\mu)$ the
irreducible highest weight representation of $\ov{\mf l}$ with
highest weight $\mu$. For $\SG=\hgl_{\dinfty}$, let $\mathcal
P^+_{\ov{\mf l}}$ be the set of $\mu\in\ov{\h}^*$ such that
$\mu=\sum_{r\in\hf\Z}\mu_r\delta_r+c\SLa^\mf{x}_0$ for some
$c\in\mathbb{C}$ and $\sum_{r\in\hf\N}\mu_r\delta_r$
(resp.~$\sum_{r\in-\hf\Z_+}\mu_r\delta_r$) is a highest weight for
an irreducible (resp.~irreducible dual) tensor representation of
$\overline{\gl}_{> 0}$ (resp. $\overline{\gl}_{\le 0}$). For
$\SG=\widehat{\mc X}$ with $\mc{X}\in\{\mc{B,C,D}\}$, let $\mathcal
P^+_{\ov{\mf l}}$ be the set of $\mu\in\ov{\h}^*$ such that
$\mu=\sum_{r\in\hf\N}\mu_r\delta_r+c\SLa^\mf{x}_0$ for some
$c\in\mathbb{C}$ and $\sum_{r\in\hf\N}\mu_r\delta_r$ is a highest
weight for an irreducible tensor representation of
$\overline{\gl}_{> 0}$.

Now consider the homology groups ${\rm H}_k(\overline{\mathfrak
u}_-;L(\SG,\SLa^\mf{x}(\la)))$, which are defined analogously (see
e.g.~\cite{Fu, KK}).

\begin{lem}\label{aux411}
The $\overline{\mathfrak l}$-module $L(\SG,\SLa^\mf{x}(\la))$ is
completely reducible.
\end{lem}

\begin{proof}
We will only show this for $\SG=\hgl_{\dinfty}$. The other cases are
analogous and omitted.

Consider the Fock space $\SF^\ell$ with $\ov{\mf{l}}\times{\rm
GL}(\ell)$-character
\begin{equation}\label{auxfockchar}
\prod_{i=1}^\ell\frac{\prod_{n\in\N}(1+x_n z_i)
(1+x_{-n+1}^{-1}z_i^{-1})}{\prod_{r\in1/2+\Z_+}
(1-x_rz_i)(1-x_{-r}^{-1}z_i^{-1})},
\end{equation}
which is the left hand side of \eqnref{character1}. Using
\eqnref{scauchy} twice we see that \eqnref{auxfockchar} can be
written as
\begin{equation*}
\sum_{\la,\mu\in\mc{P}^+;l(\la),l(\mu)\le \ell}{HS}_{\la'}({\bf
x}_{>0}){HS}_{\mu}({\bf x}^{-1}_{\le 0})
s_{\la}(z_1,\cdots,z_\ell)s_\mu(z_1^{-1},\cdots,z_\ell^{-1}).
\end{equation*}
Thus \eqnref{auxfockchar} can be written as an infinite sum of
${HS}_{\la'}({\bf x}_{>0}){HS}_{\mu}({\bf x}^{-1}_{\le 0})$ such
that each summand has a finite multiplicity. Now \thmref{catnon:ss}
implies that $\SF^\ell$, as an $\ov{\mf{l}}$-module, is completely
reducible. Since $L(\SG,\SLa^{\mf{a}}(\la))$ is a direct summand of
a completely reducible module, it is also completely reducible.
\end{proof}

\begin{lem}\label{complete:reducibility}
The $\overline{\mathfrak l}$-module ${\rm H}_k(\overline{\mathfrak
u}_-;L(\SG,\SLa^\mf{x}(\la)))$ is completely reducible.
\end{lem}

\begin{proof}
First suppose that $\SG=\hgl_{\dinfty}$ with
$\ov{\mf{l}}=\overline{\gl}_{\le 0}\oplus\overline{\gl}_{>0}\oplus\C
K$. As a $\overline{\gl}_{\le 0}\oplus\overline{\gl}_{>0}$-module,
$\bar{\mathfrak u}_-\cong\C^{\le 0 *}\otimes\C^{>0}$. We claim that
$\La^k(\bar{\mathfrak u}_-)$ is completely reducible, and each
irreducible component is an irreducible dual tensor representation
of $\overline{\gl}_{\le 0}$ tensored with an irreducible tensor
representation of $\overline{\gl}_{> 0}$. This can be proved, for
example, by using an analogous truncation functor argument as in
\secref{standardborel} in combination with the skew-symmetric Howe
duality for a pair of finite-dimensional general Lie superalgebras
\cite[Theorem 3.3]{CW1}. By \lemref{aux411}
$L(\hgl_{\infty\vert\infty},\SLa^{\mf a}(\la))$ as an
$\overline{\mathfrak l}$-module is completely reducible, and hence
by \thmref{catnon:ss} $\Lambda^k(\bar{\mathfrak u}_-)\otimes
L(\hgl_{\infty\vert\infty},\SLa^{\mf a}(\la))$ is a completely
reducible $\bar{\mathfrak l}$-module.  Since any subquotient of a
completely reducible module is also completely reducible, the result
follows.

Now suppose that $\SG=\widehat{\mc B},\widehat{\mc C},\widehat{\mc
D}$. Then $\ov{\mf{l}}$ is isomorphic to $\ov{\gl}_{>0}\oplus\C K$
and as a $\ov{\gl}_{>0}$-module,
\begin{align}\label{u-}
\ov{\mf u}_-\cong
\begin{cases}
\La^2(\C^{>0})\oplus \C^{>0}, & \text{if $\SG=\widehat{\mc
B}$},\\
S^2(\C^{>0}), & \text{if $\SG=\widehat{\mc
C}$},\\
\La^2(\C^{>0}), & \text{if $\SG=\widehat{\mc D}$},
\end{cases}
\end{align}
which is a direct sum of irreducible tensor representations. By
\thmref{catnon:ss} $\La^k(\ov{\U}_-)$ is completely reducible over
$\ov{\mf{l}}$. This together with \lemref{aux411} shows that
$\Lambda^k(\bar{\mathfrak u}_-)\otimes L(\ov{\G},\SLa^\mf{x}(\la))$,
and hence ${\rm H}_k(\overline{\mathfrak
u}_-;L(\SG,\SLa^\mf{x}(\la)))$, is completely reducible.
\end{proof}

\begin{rem}
Alternatively, Lemmas \ref{aux411} and \ref{complete:reducibility}
also follow from the unitarity of the respective
$\ov{\mf{l}}$-modules.
\end{rem}

We have the following super-analogue of the action of the Casimir
operator on homology groups (cf.~\cite{GL, L, J}).

\begin{prop}\label{eigenvalue:super}
Let $\gamma\in\mathcal P^+_{\bar{\mathfrak l}}$ and let
$L(\bar{\mathfrak l},\gamma)$ be the irreducible $\bar{\mathfrak
l}$-module of highest weight $\gamma$.  If $L(\bar{\mathfrak
l},\gamma)$ is a component of ${\rm H}_k(\overline{\mathfrak
u}_-;L(\SG,\SLa^\mf{x}(\la)))$, then
$(\gamma+2\rho_s\vert\gamma)_s=(\SLa^\mf{x}(\la)+2\rho_s\vert\SLa^\mf{x}(\la))_s$.
\end{prop}

\begin{proof}
This can be proved following the same type of arguments as the one
given in the proof of \cite[Proposition 18]{L} and thus we will omit
the details. We only remark that in the process we use the same
bilinear form $(\cdot\vert\cdot)_s$ and the same $\rho_s$ to define
the corresponding Casimir operator for $\ov{\mf{l}}$ as in
\eqnref{supercasimir}.
\end{proof}

\section{Computation of $\overline{\mathfrak u}_-$-homology
groups}\label{sec:homology}

\subsection{Comparison of Casimir
eigenvalues}\label{comp:eigenvalues}

Let $\SP^1$ denote the set of sequences
$a=(a_\hf,a_1,a_{\frac{3}{2}},\cdots)$, where
$(a_{\hf},a_{\frac{3}{2}},\cdots)$ and $(a_1,a_2,\cdots)$ are strict
partitions, such that $a_s=0$ implies that $a_{s+\hf}=0$, for all
$s$. We can define a map $\theta_1:\mathcal P^+\rightarrow\SP^1$ as
follows. For $\la=(\la_1,\la_2,\cdots )\in\mathcal P^+$, let
\begin{equation*}
\theta_1(\la)=(\langle\la'_1\rangle,\langle\la_1-1\rangle,
\langle\la'_2-1\rangle,\langle\la_2-2\rangle,\langle\la'_3-2\rangle,\langle\la_3-3\rangle,
\cdots).
\end{equation*}
It is easy to see that $\theta_1(\la)\in\SP^1$ and $\theta_1$ is a
bijection. We have the following combinatorial lemma that is
equivalent to \cite[1.7]{M}, where an elegant proof can be found.

\begin{lem}\label{aux111} Let $i=1,\cdots,N$ with $N\ge \la'_1$.
The set $\{\la'_i-i+1\,\vert\, \la'_i-i+1>0\}\cup\{-\la_i+i
\,\vert\, \la_i-i<0\}$ is a permutation of the set
$\{1,2,\cdots,N\}$.
\end{lem}

For $a,b\in\SP^1$, we define
\begin{equation*}
(a\vert b)_s:=\sum_{r\in\hf\N}(-1)^{2r}a_rb_r.
\end{equation*}
For $\la\in\mc{P}^+$, we define
$$(\la+2\rho_1\vert\la)_{1}:=\sum_{i\in\N} \la_i(\la_i-2i).$$

\begin{prop}\label{prop:key}
For $\la\in\mathcal P^+$, we have
\begin{equation*}
(\la+2\rho_1\vert\la)_{1}=(\theta_1(\la)\vert\theta_1(\la))_s.
\end{equation*}
\end{prop}

\begin{proof}
Let $s$ be the smallest positive integer such that $\la_s-s\le 0$,
and $l$ the length of $\la$. We have {\allowdisplaybreaks
\begin{align*}
(\theta_1(\la)\vert\theta_1(\la))_s&=\sum_{i}\langle\la_i-i\rangle^2-\sum_j\langle\la'_j-j+1\rangle^2\\
&=\sum_{i=1}^l(\la_i-i)(\la_i-i)-\sum_{j=s}^l(\la_i-i)^2-\sum_j\langle\la'_j-j+1\rangle^2\\
&=\sum_{i=1}^l(\la_i-i)(\la_i-i)-\sum_{j=s}^l(-\la_i+i)^2-\sum_j\langle\la'_j-j+1\rangle^2\\
&=\sum_{i=1}^l(\la_i-i)(\la_i-i)-\sum_{j=1}^{\la'_1}j^2\quad{\rm by\ \lemref{aux111}}\\
&=\sum_{i=1}^l\la_i(\la_i-2i)+\sum_{i=1}^l i^2-\sum_{j=1}^{\la'_1}j^2\\
&=\sum_{i=1}^l\la_i(\la_i-2i)=(\la+2\rho_1\vert\la)_{1}.
\end{align*}}
\end{proof}

We also need to consider a slightly different setup. Let $\SP^2$
denote the set of sequences
$a=(a_0,a_\hf,a_1,a_{\frac{3}{2}},\cdots)$, where
$(a_0,a_1,a_2,\cdots)$ and
$(a_\hf,a_{\frac{3}{2}},a_{\frac{5}{2}},\cdots)$ are strict
partitions, such that $a_s=0$ implies $a_{s+\hf}=0$ for all $s$. Let
$\la$ be a partition, whose rows and columns we index with $\Z_+$,
i.e.~$\la=(\la_0,\la_1,\cdots )$. We define similarly an element
$\theta_2(\la)\in\SP^2$ by
\begin{equation*}
\theta_2(\la):=(\langle\la_0\rangle,\langle\la'_0-1\rangle,\langle\la_1-1\rangle,
\langle\la'_1-2\rangle,\langle\la_2-2\rangle,\langle\la'_2-3\rangle,
\cdots).
\end{equation*}
For $a,b\in\SP^2$, we define
\begin{equation*}
(a\vert b)_s:=\sum_{r\in\hf\Z_+}(-1)^{2r}a_rb_r.
\end{equation*}
For $\la\in\mc{P}^+$, we define
$(\la+2\rho_2\vert\la)_{2}:=\sum_{i\in\Z_+}\la_i(\la_i-2i)$.
Similarly one can show the following.

\begin{prop}\label{prop:key2}
For $\la\in\mathcal P^+$, we have
\begin{equation*}
(\la+2\rho_2\vert\la)_{2}=(\theta_2(\la)\vert\theta_2(\la))_s.
\end{equation*}
\end{prop}

\begin{cor}\label{samecasimir}
Let $\la,\mu$ be partitions.  For $i=1,2$, we have
$(\la+2\rho_i\vert\la)_{i}=(\mu+2\rho_i\vert\mu)_{i}$ if and only if
$(\theta_i(\la)\vert\theta_i(\la))_s=(\theta_i(\mu)\vert\theta_i(\mu))_s$.
\end{cor}

\subsection{Formulas for the $\overline{\mathfrak u}_-$-homology groups}
We keep the notation of Section 4. Suppose that
$\mu=\sum_{i\in\N}\mu_i\epsilon_i-\sum_{j\in-\Z_+}\mu_j\epsilon_j+n\La^\mf{x}_0\in\mc{P}_{\mf
l}^+$ is given, where we assume that $\mu_j=0$ for $j\leq 0$ if
$\G\neq \hgl_{\infty}$. By definition, $\mu^+=(\mu_1,\mu_2,\cdots)$
and $\mu^-=(\mu_0,\mu_{-1},\mu_{-2},\cdots)$ are partitions. Set
$\theta_1(\mu^+)=(a_{\hf},a_1,a_{\frac{3}{2}},\cdots)\in\SP^1$ and
$\theta_2(\mu^-)=(b_0,b_\hf,b_1,b_{\frac{3}{2}},\cdots)\in\SP^2$. We
define
\begin{equation}\label{thetamap}
\theta(\mu):=
\begin{cases}
\sum_{r\in\hf\N}a_r\delta_r-
\sum_{s\in\hf\Z_+}b_s\delta_{-s}+\ov{n}\SLa^{\mf a}_0, & \text{if $\G=\hgl_{\infty}$}, \\
\sum_{r\in\hf\N}a_r\delta_r+\ov{n}\SLa^\mf{x}_0, & \text{if
$\G=\mf{b}_{\infty},\mf{c}_{\infty},\mf{d}_{\infty}$},
\end{cases}
\end{equation}
where $\ov{n}=n\langle \La^\mf{x}_0,K\rangle$. It is not difficult
to see that $\theta(\mu) \in{\mathcal P}^+_{\bar{\mathfrak l}}$ and
$\theta$ yields a bijection from $\mc{P}_{\mf l}^+$ to
$\mc{P}_{\ov{\mf{l}}}^+$. In particular,
$\theta(\La^\mf{x}(\la))=\SLa^\mf{x}(\la)$ for $\la\in \mc{P}(G)$.

For $\la\in \mc{P}(G)$, we put
\begin{equation}\label{def:L}
\mathrm{L}(\G,\La^\mf{x}(\la)):=
\begin{cases}
L(\G,\La^{\mf d}(\la))\oplus L(\G,\La^{\mf d}(\tilde{\la})), & \text{if $G=\rm{O}(2\ell)$}, \\
L(\G,\La^\mf{x}(\la)), & \text{otherwise}.
\end{cases}
\end{equation}
We define $\mathrm{L}(\SG,\SLa^\mf{x}(\la))$ in a similar way. For
$w\in W^0$ we put
\begin{equation}\label{def:LL} \mathrm{L}(\ov{\mf{l}},\theta(w\circ\La^\mf{x}(\la))):=
\begin{cases}
L(\ov{\mf{l}},\theta(w\circ\La^{
\mf d}(\la)))\oplus L(\ov{\mf{l}},\theta(w\circ\La^{\mf d}(\tilde{\la}))), & \text{if $G=\rm{O}(2\ell)$}, \\
L(\ov{\mf{l}},\theta(w\circ\La^\mf{x}(\la))), & \text{otherwise}.
\end{cases}
\end{equation}
Also $\mathrm{L}(\mf{l},w\circ\La^\mf{x}(\la))$ is defined
similarly.

Note that ${\rm
ch}\mathrm{L}(\SG,\SLa^\mf{x}(\la))=\omega^\mf{x}\left( {\rm
ch}\mathrm{L}(\G,\La^\mf{x}(\la))\right)$ and ${\rm ch}{L}(\ov{\mf
l},\theta(\eta))=\omega^\mf{x}\left( {\rm
ch}{L}(\mf{l},\eta)\right)$, for $\eta\in\mc{P}^+_{\mf l}$.

\begin{lem}\label{aux112}
Let $\la\in\mc{P}(G)$ and $\bar{\mu}\in{\mathcal
P}^+_{\bar{\mathfrak l}}$. If $L(\bar{\mathfrak l},\bar{\mu})$
appears in $\Lambda^k\bar{{\mathfrak u}}_-\otimes
\mathrm{L}(\SG,\SLa^\mf{x}(\la))$, then there exists a unique
$\mu\in\mathcal P^+_{\mathfrak l}$, with $\theta(\mu)=\bar{\mu}$,
such that $L(\mathfrak l,\mu)$ appears in $\Lambda^k\mathfrak
u_-\otimes \mathrm{L}(\G,\La^\mf{x}(\la))$ with the same
multiplicity.
\end{lem}

\begin{proof} Since we have $\omega^\mf{x}\left({\rm
ch}\mathrm{L}(\G,\La^\mf{x}(\la))\right)={\rm
ch}\mathrm{L}(\SG,\ov{\La}^\mf{x}(\la))$ for $\la\in\mc{P}(G)$, we
conclude that
\begin{equation*}
\omega^\mf{x}\big{(}{\rm ch}\big{[}\Lambda^k\mathfrak u_-\otimes
\mathrm{L}(\G,\La^\mf{x}(\la))\big{]}\big{)}={\rm
ch}\big{[}\Lambda^k\bar{{\mathfrak u}}_-\otimes
\mathrm{L}(\SG,\SLa^\mf{x}(\la))\big{]}.
\end{equation*}

Since $\theta$ is a bijection, there exists a unique
$\mu\in\mc{P}^+_{\mf l}$ such that $\theta(\mu)=\ov{\mu}$. Therefore
$L(\mathfrak l,\mu)$ is a composition factor of
 $\Lambda^k\mathfrak u_-\otimes
\mathrm{L}(\G,\La^\mf{x}(\la))$ if and only if $L(\bar{\mathfrak
l},\theta(\mu))$ is a composition factor of
$\Lambda^k\bar{{\mathfrak u}}_-\otimes
\mathrm{L}(\SG,\SLa^\mf{x}(\la))$. Furthermore they have the same
multiplicity.
\end{proof}

The following lemma is crucial in the sequel.

\begin{lem}\label{aux113}
For ${\mu}\in{\mathcal P}^+_{{\mathfrak l}}$, we have
$$(\mu+2\rho_c|\mu)_c
=(\theta(\mu)+2\rho_s|\theta(\mu))_s.$$ In particular, for $\la\in
\mc{P}(G)$, we have
$(\La^\mf{x}(\la)+2\rho_c\vert\La^\mf{x}(\la))_c=(\SLa^\mf{x}(\la)+2\rho_s\vert\SLa^\mf{x}(\la))_s$.
\end{lem}

\begin{proof}
\textit{Case 1.} $\G=\hgl_{\infty}$. Suppose that
$\mu=\sum_{i\in\N}\mu_i\epsilon_i-\sum_{i\in-\Z_+}\mu_i\epsilon_i+n\La^{\mf
a}_0$, where $\mu^+=(\mu_1,\mu_2,\ldots)$ and
$\mu^-=(\mu_0,\mu_{-1},\ldots)$ are partitions. Let
$\theta_1(\mu^+)=(a_1,a_{\frac{3}{2}},\ldots)$ and
$\theta_2(\mu^-)=(b_0,b_{\frac{1}{2}},\ldots)$. Using
\propref{prop:key} and \propref{prop:key2}, we compute:
{\allowdisplaybreaks
\begin{align*}
&(\mu+2\rho_c\vert\mu)_c \\
&=(\sum_{i\in\N}\mu_i\epsilon_i+2\rho_c\vert\sum_{i\in\N}\mu_i\epsilon_i)_{c}+
(\sum_{i\in-\Z_+}\mu_i\epsilon_i-2\rho_c\vert\sum_{i\in-\Z_+}\mu_i\epsilon_i)_{c}
\\&\hskip 4cm+2n(\La^{\mf a}_0\vert\sum_{i\in\N}\mu_i\epsilon_i)_c-2n(\La^{\mf a}_0\vert\sum_{i\in-\Z_+}\mu_i\epsilon_i)_c\\
&=\sum_{i\in\N} \mu_i(\mu_i-2i)+\sum_{i\in\Z_+}
\mu_{-i}(\mu_{-i}-2i) -n(\sum_{i\in\N}\mu_i+\sum_{i\in-\Z_+}\mu_i)\\
&=(\mu^++2\rho_1\vert\mu^+)_{1}+
(\mu^-+2\rho_2\vert\mu^-)_{2}-n(\sum_{i\in\N}\mu_i+\sum_{i\in-\Z_+}\mu_i)\\
&=(\theta_1(\mu^+)\vert\theta_1(\mu^+))_s+({\theta}_2(\mu^-)\vert{\theta_2}(\mu^-))_{s}
-n(\sum_{r\in\hf\N}a_r+\sum_{r
\in\hf\Z_+}b_r)\\
&=(\sum_{r\in\hf\N}a_r\delta_r\vert\sum_{r\in\hf\N}a_r\delta_r)_s+(\sum_{r
\in\hf\Z_+}b_r\delta_r\vert\sum_{r \in\hf\Z_+}b_r\delta_r)_{s} \\
& \hskip 4cm +2n(\SLa^{\mf
a}_0\vert\sum_{r\in\hf\N}a_r\delta_r)-2n(\SLa^{\mf a}_0\vert\sum_{r
\in\hf\Z_+}b_r\delta_r)_s\\
&=(\theta(\mu)\vert\theta(\mu))_s=(\theta(\mu)+2\rho_s\vert\theta(\mu))_s,
\end{align*}}
where $\rho_s=0$ in this case.\vskip 3mm

Next, we assume that $\G=\mf{x}_{\infty}$ with
$\mf{x}\in\{\mf{b,c,d}\}$. Suppose that
$\mu=\sum_{i\in\N}\mu_i\epsilon_i+n\La^\mf{x}_0$, where
$\mu^{\circ}=(\mu_1,\mu_2,\ldots)$ is a partition. Let
$\theta_1(\mu^{\circ})=(a_1,a_{\frac{3}{2}},\ldots)$.

\textit{Case 2.} $\G=\mf{b}_{\infty}$. {\allowdisplaybreaks
\begin{align*}
&(\mu+2\rho_c\vert\mu)_c \\ &=
(\sum_{i\in\N}\mu_i\epsilon_i+2\rho_c\vert\sum_{i\in\N}\mu_i\epsilon_i)_{c}
+2n(\La^{\mf b}_0\vert\sum_{i\in\N}\mu_i\epsilon_i)_c=\sum_{i\in\N}\mu_i(\mu_i-2i+1) -n\sum_{i\in\N}\mu_i \\
&=(\mu^{\circ}+\rho_1|\mu^{\circ})_1
+\sum_{i\in\N}\mu_i-n\sum_{i\in\N}\mu_i
=(\theta_1(\mu^{\circ})\vert\theta_1(\mu^{\circ}))_s
+\sum_{r\in\hf\N}a_r-n\sum_{r\in\hf\N}a_r\\
&=(\sum_{r\in\hf\N}a_r\delta_r\vert\sum_{r\in\hf\N}a_r\delta_r)_s
+2(\rho_s\vert\sum_{r\in\hf\N}a_r\delta_r)_s+n(\SLa^{\mf
b}_0\vert\sum_{r\in\hf\N}a_r\delta_r)_s
\\&= (\sum_{r\in\hf\N}a_r\delta_r+\hf n\SLa^{\mf b}_0+2\rho_s\vert\sum_{r\in\hf\N}a_r\delta_r+\hf n\SLa^{\mf
b}_0)_s=(\theta(\mu)+2\rho_s\vert\theta(\mu))_s.
\end{align*}}\vskip 3mm

\textit{Case 3.} $\G=\mf{c}_{\infty}$. {\allowdisplaybreaks
\begin{align*}
&(\mu+2\rho_c\vert\mu)_c \\ &=
(\sum_{i\in\N}\mu_i\epsilon_i+2\rho_c\vert\sum_{i\in\N}\mu_i\epsilon_i)_{c}
+2n(\La^{\mf c}_0\vert\sum_{i\in\N}\mu_i\epsilon_i)_c=\sum_{i\in\N}\mu_i(\mu_i-2i) -2n\sum_{i\in\N}\mu_i \\
&=(\mu^{\circ}+\rho_1|\mu^{\circ})_1 -2n\sum_{i\in\N}\mu_i
=(\theta_1(\mu^{\circ})\vert\theta_1(\mu^{\circ}))_s
-2n\sum_{r\in\hf\N}a_r
\\
&=(\sum_{r\in\hf\N}a_r\delta_r\vert\sum_{r\in\hf\N}a_r\delta_r)_s
+2n(\SLa^{\mf c}_0\vert\sum_{r\in\hf\N}a_r\delta_r)_s
\\&= (\sum_{r\in\hf\N}a_r\delta_r+n\SLa^{\mf c}_0+2\rho_s\vert\sum_{r\in\hf\N}a_r\delta_r+ n\SLa^{\mf
c}_0)_s=(\theta(\mu)+2\rho_s\vert\theta(\mu))_s.
\end{align*}}
Note that $\rho_s=0$ in this case. \vskip 3mm

\textit{Case 4.} $\G=\mf{d}_{\infty}$. {\allowdisplaybreaks
\begin{align*}
&(\mu+2\rho_c\vert\mu)_c \\ &=
(\sum_{i\in\N}\mu_i\epsilon_i+2\rho_c\vert\sum_{i\in\N}\mu_i\epsilon_i)_{c}
+2n(\La^{\mf d}_0\vert\sum_{i\in\N}\mu_i\epsilon_i)_c=\sum_{i\in\N}\mu_i(\mu_i-2(i-1)) -n\sum_{i\in\N}\mu_i \\
&=(\mu^{\circ}+\rho_1|\mu^{\circ})_1
+2\sum_{i\in\N}\mu_i-n\sum_{i\in\N}\mu_i
=(\theta_1(\mu^{\circ})\vert\theta_1(\mu^{\circ}))_s
+2\sum_{r\in\hf\N}a_r-n\sum_{r\in\hf\N}a_r\\
&=(\sum_{r\in\hf\N}a_r\delta_r\vert\sum_{r\in\hf\N}a_r\delta_r)_s
+2(\rho_s\vert\sum_{r\in\hf\N}a_r\delta_r)_s+n(\SLa^{\mf
b}_0\vert\sum_{r\in\hf\N}a_r\delta_r)_s
\\&= (\sum_{r\in\hf\N}a_r\delta_r+\hf n\SLa^{\mf d}_0+2\rho_s\vert\sum_{r\in\hf\N}a_r\delta_r+\hf n\SLa^{\mf
d}_0)_s=(\theta(\mu)+2\rho_s\vert\theta(\mu))_s.
\end{align*}}
\end{proof}

\begin{rem}
Note that for $\la\in\mc{P}({\rm O}(2\ell))$
$$(\La^{\mf d}(\la)+2\rho_c\vert\La^{\mf d}(\la))_c=(\La^{\mf d}(\tilde{\la})+2\rho_c\vert\La^{\mf d}(\tilde{\la}))_c.$$
\end{rem}

\begin{cor}\label{aux222}
Let $\la\in\mc{P}(G)$ and let
$\La^k(\SU_-)\otimes\mathrm{L}(\SG,\SLa^\mf{x}(\la))=
\bigoplus_{\eta\in\mc{P}^+_{\ov{\mf{l}}}}L(\ov{\mf{l}},\eta)^{m_\eta}$
be its decomposition into irreducible $\ov{\mf{l}}$-modules. Then
$\eta=\theta(w\circ\La^\mf{x}(\la))$, for some $w\in W^0_k$, if and
only if
$(\eta+2\rho_s\vert\eta)_s=(\SLa^\mf{x}(\la)+2\rho_s\vert\SLa^\mf{x}(\la))_s$.
\end{cor}

\begin{proof}
Suppose that there exits $w\in W^0_k$ with
$\eta=\theta(w\circ\La^\mf{x}(\la))$. Then by \lemref{aux113} and
the $W$-invariance of $(\cdot\vert\cdot)_c$ we have
\begin{equation*}
(\eta+2\rho_s\vert\eta)_s=(w\circ\La^\mf{x}(\la)+2\rho_c\vert
w\circ\La^\mf{x}(\la))_c=(\La^\mf{x}(\la)+2\rho_c\vert
\La^\mf{x}(\la))_c= (\SLa^\mf{x}(\la)+2\rho_s\vert
\SLa^\mf{x}(\la))_s.
\end{equation*}
Conversely, suppose that
$(\eta+2\rho_s\vert\eta)_s=(\SLa^\mf{x}(\la)+2\rho_s\vert\SLa^\mf{x}(\la))_s$.
By \lemref{aux113} we have
\begin{equation*}
(\theta^{-1}(\eta)+2\rho_c\vert\theta^{-1}(\eta))_c=(\La^\mf{x}(\la)+2\rho_c\vert\La^\mf{x}(\la))_c.
\end{equation*}
By \propref{eigenvalue:classical} (i) there exist $w\in W^0_k$ such
that $\theta^{-1}(\eta)=w\circ\La^\mf{x}(\la)$.
\end{proof}

\begin{rem}\label{multirem}
The multiplicity $m_\eta$ equals $1$ if $\eta$ satisfies the
condition of \corref{aux222} for $\mf{x}=\mf{a,b,c}$ or
$\mf{x}=\mf{d}$ with $\la\not=\tilde{\la}$. In the case of
$\mf{x}=\mf{d}$ with $\la=\tilde{\la}$ we have $m_\eta=2$. This
follows from \eqnref{def:L}, \propref{eigenvalue:classical} and
\lemref{aux112}.
\end{rem}

\begin{thm}\label{mainthm}
Let $\la\in\mc{P}(G)$, $k\in\Z_+$, and $\theta$ as in {\rm
\eqnref{thetamap}}. As $\ov{\mf{l}}$-modules we have
\begin{equation*}
{\rm
H}_k(\ov{\mf{u}}_-;\mathrm{L}(\SG,\SLa^\mf{x}(\la)))\cong\bigoplus_{w\in
W^0_k}\mathrm{L}(\ov{\mf{l}},\theta(w\circ\La^\mf{x}(\la))).
\end{equation*}
 In particular, ${\rm ch}\big{[}{\rm H}_k(\bar{\mathfrak
u}_-;\mathrm{L}(\SG,\SLa^\mf{x}(\la)))\big{]}=\omega^\mf{x}\big{(}{\rm
ch}\big{[}{\rm H}_k({\mathfrak
u}_-;\mathrm{L}(\G,\La^\mf{x}(\la)))\big{]}\big{)}$.
\end{thm}

\begin{proof} Let $\mu\in\mathcal P^+_{\mathfrak l}$ be such that $L(\mathfrak l,\mu)$
is a composition factor of ${\rm H}_k({\mathfrak
u}_-;\mathrm{L}(\G,\La^\mf{x}(\la)))$. Then it is precisely a
composition factor of $\Lambda^k\mathfrak u_-\otimes
\mathrm{L}(\G,\La^\mf{x}(\la))$ with
$(\mu+2\rho_c\vert\mu)_c=(\La^\mf{x}(\la)+2\rho_c\vert\La^\mf{x}(\la))_c$
by \propref{eigenvalue:classical}. Furthermore each appears with
multiplicity given in \remref{multirem}. By \corref{aux222} the
corresponding $\theta(\mu)$'s are precisely the weights in
$\mc{P}^+_{\ov{\mf l}}$ such that $L(\ov{\mf{l}},\theta(\mu))$
appears as a composition factor of $\Lambda^k\bar{{\mathfrak
u}}_-\otimes \mathrm{L}(\SG,\SLa^\mf{x}(\la))$ with
$(\theta(\mu)+2\rho_s\vert\theta(\mu))_s=(\SLa^\mf{x}(\la)+2\rho_s\vert\SLa^\mf{x}(\la))_s$,
which has the same multiplicity.

Now, Theorems \ref{char:super-A}--\ref{char:super-D}, and the
Euler-Poincar\'e principle imply
\begin{align*}
&\sum_{k=0}^\infty(-1)^k{\rm ch}\big{[}{\rm H}_k(\bar{\mathfrak
u}_-;\mathrm{L}(\SG,\SLa^\mf{x}(\la)))\big{]}=\sum_{k=0}^\infty\sum_{w\in
W^0_k}(-1)^k \mc{H}_{\la,w},
\end{align*}
where
\begin{equation*}
\mc{H}_{\la,w}=
\begin{cases}
HS_{(\la^+_w)'}({\bf x}_{>0}) HS_{\la^-_w}({\bf x}^{-1}_{\le 0}), &
\text{if
$G={\rm GL}(\ell)$}, \\
HS_{(\la_w)'}({\bf x}_{>0}), & \text{if
$G={\rm Pin}(2\ell), {\rm Sp}(2\ell), {\rm O}(2\ell+1)$}, \\
HS_{(\la_w)'}({\bf x}_{>0})+HS_{(\tilde{\la}_w)'}({\bf x}_{>0}), &
\text{if $G={\rm O}(2\ell)$}.
\end{cases}
\end{equation*}
Since all the highest weights are distinct, we conclude from
Proposition \ref{eigenvalue:super} that
\begin{align*}
{\rm ch}\big{[}{\rm H}_k(\bar{\mathfrak
u}_-;\mathrm{L}(\SG,\SLa^\mf{x}(\la)))\big{]}=\sum_{w\in
W^0_k}\mc{H}_{\la,w}.
\end{align*}
The right hand side is equal to $\omega^\mf{x}\big{(}{\rm
ch}\big{[}{\rm H}_k({\mathfrak
u}_-;\mathrm{L}(\G,\La^\mf{x}(\la)))\big{]}\big{)}$ by
\propref{homologyclassical}.
\end{proof}

\thmref{mainthm} and \corref{aux222} imply the following super
analogue of \propref{eigenvalue:classical} (ii), which is the
converse of Proposition of \ref{eigenvalue:super}.

\begin{cor}
If $L(\overline{\mathfrak l},\eta)$ is an irreducible
$\overline{\mathfrak l}$-module in $\Lambda^k(\overline{\mathfrak
u}_-)\otimes \mathrm{L}(\SG,\SLa^\mf{x}(\la))$ such that $(\eta
+2\rho_s\vert\eta)_s=(\SLa^\mf{x}(\la)+2\rho_s\vert\SLa^\mf{x}(\la))_s$,
then $L(\overline{\mathfrak l},\eta)$ is a component of ${\rm
H}_k(\overline{\mathfrak u}_-;\mathrm{ L}(\SG,\SLa^\mf{x}(\la)))$.
Furthermore the multiplicities coincide.
\end{cor}

\begin{rem} In the spirit of \cite{CWZ} we can use
Serganova's homological interpretation of Kazhdan-Lusztig
polynomials \cite{Se} to reformulate \thmref{mainthm} into a
statement about equality of certain Kazhdan-Lusztig polynomials for
$\G$ and $\SG$.
\end{rem}

\section{Resolution in terms of generalized Verma
modules}\label{sec:resolution}

Let $V$ be a $\SG$-module, on which the action of $\ov{\mf u}_+$ is
locally nilpotent. For $c\in\C$, we define
\begin{equation*}
V^c:=\{v\in V\vert (\overline{\Omega}-c)^nv=0, n\gg 0\},
\end{equation*}
i.e.~$V^c$ is the generalized $\overline{\Omega}$-eigenspace
corresponding to the eigenvalue $c$.  Clearly we have
$V=\bigoplus_{c\in\C}V^c$.

We can now construct a resolution of generalized Verma modules for
$\mathrm{L}(\SG,\SLa^\mf{x}(\la))$.

Note that $\SG/\sP$ is a $\sP$-module, on which $\ov{\mf{u}}_+$ acts
trivially, and hence so is $\Lambda^k(\SG/\sP)$, for all $k\in\Z_+$.
Thus the induced module $U(\SG)\otimes_{\sP}\Lambda^k(\SG/\sP)$ is a
$\SG$-module. For $\la\in\mc{P}(G)$ we consider the $\SG$-module
$Y_k=\left(U(\SG)\otimes_{U(\overline{\mathfrak
p})}\Lambda^k(\SG/\overline{\mathfrak p})\right)\otimes \mathrm{
L}\left(\SG,\SLa^\mf{x}(\la)\right)$ . Let
\begin{equation}\label{standardres}
\cdots\stackrel{d_{k+1}}{\longrightarrow}Y_k\stackrel{d_k}{\longrightarrow}Y_{k-1}\stackrel{d_{k-1}}{\longrightarrow}\cdots
\stackrel{d_1}{\longrightarrow}Y_0\stackrel{d_0}{\longrightarrow}\mathrm{
L}(\SG,\SLa^\mf{x}(\la))\rightarrow 0
\end{equation}
be the standard resolution of ${\rm L}(\SG,\SLa^\mf{x}(\la))$ as in
\cite{KK} (cf.~\cite{GL, J, L}). Put
$$c_{\la}=(\SLa^\mf{x}(\la)+2\rho_s|\SLa^\mf{x}(\la))_s.$$
The restriction of {\rm \eqnref{standardres}} to the generalized
$c_{\la}$-eigenspace of $\overline{\Omega}$ produces a resolution of
$\SG$-modules
\begin{equation}\label{resolution:verma}
\cdots\stackrel{d_{k+1}}{\longrightarrow}Y^{c_{\la}}_k\stackrel{d_k}
{\longrightarrow}Y^{c_{\la}}_{k-1}\stackrel{d_{k-1}}{\longrightarrow}\cdots
\stackrel{d_1}{\longrightarrow}Y^{c_{\la}}_0\stackrel{d_0}
{\longrightarrow}\mathrm{L}(\SG,\SLa^\mf{x}(\la))\rightarrow 0.
\end{equation}

For $\nu\in \mc{P}^+_{\ov{\mf{l}}}$, let
$V(\nu)=U(\SG)\otimes_{U(\ov{\mf p})}L(\ov{\mf l},\nu)$ be the
generalized Verma module, where $\ov{\mf u}_+$ acts trivially on
$L(\ov{\mf l},\nu)$. We have the following counterpart of Theorem
8.7 of \cite{GL}.

\begin{thm}\label{resolution}
For $\la\in\mc{P}(G)$, the {\rm \eqnref{resolution:verma}} is a
resolution of $\mathrm{ L}(\SG,\SLa^\mf{x}(\la))$ such that each
$Y^{c_{\la}}_k$ has a flag of generalized Verma modules. Furthermore
as $\ov{\mf{l}}+\ov{\mf{u}}_-$-modules we have
$$Y^{c_{\la}}_k\cong
\begin{cases}
\bigoplus_{w\in W^0_k} V(\theta(w\circ\La^\mf{x}(\la))), & \text{if
$G\neq \rm{O}(2\ell)$}, \\
\bigoplus_{w\in W^0_k}\left[ V(\theta(w\circ\La^{\mf d}(\la)))\oplus
V(\theta(w\circ\La^{\mf d}(\tilde{\la})))\right], & \text{if
$G=\rm{O}(2\ell)$}.
\end{cases}
$$
\end{thm}

\begin{proof}
To keep notation simple let us assume that $\G\neq \mf{d}_\infty$.

First we claim that $Y_k\cong U(\SG)\otimes_{U(\overline{\mathfrak
p})}(\Lambda^k(\SG/\overline{\mathfrak p})\otimes
{L}(\SG,\SLa^\mf{x}(\la)))$ has a flag of generalized Verma modules.
This can be seen as follows. From the proof of
\lemref{complete:reducibility} we see that the space
$\Lambda^k(\SG/\overline{\mathfrak p})$ decomposes into a finite
direct sum of $\ov{\mf p}$-modules, which is semisimple over
$\ov{\mf l}$ and on which $\ov{\mf u}_+$ acts trivially. On the
other hand ${L}(\SG,\SLa^\mf{x}(\la))$ has a $\ov{\mf p}$-submodule,
which as an $\ov{\mf l}$-module is isomorphic to $L(\ov{\mf
l},\SLa^\mf{x}(\la))$ and on which $\overline{\mathfrak u}_+$ acts
trivially.  This module is precisely the $\overline{\mathfrak
u}_+$-invariants of ${L}(\SG,\SLa^\mf{x}(\la))$. Now we consider the
$\overline{\mathfrak u}_+$-invariants of
${L}(\SG,\SLa^\mf{x}(\la))/L(\overline{\mf l},\SLa^\mf{x}(\la))$,
which is a direct sum of irreducible $\sL$-modules. Continuing this
way we see that ${L}(\SG,\SLa^\mf{x}(\la))$ as a $\ov{\mf p}$-module
has an infinite filtration of $\ov{\mf p}$-modules of the form
\begin{equation*}
L_0={L}(\SG,\SLa^\mf{x}(\la))\supseteq L_1\supseteq
L_2\supseteq\cdots
\end{equation*}
such that $L_i/L_{i+1}$ is a direct sum of irreducible $\sL$-modules
on which $\ov{\mf u}_+$ acts trivially. Thus the same is true for
$\Lambda^k(\SG/\overline{\mathfrak p})\otimes
{L}(\SG,\SLa^\mf{x}(\la))$. Now $Y_k$ is the induced module and
hence the claim follows.  Since $Y_k^{c_\la}$ is a direct summand of
$Y_k$, it also has a flag of generalized Verma modules.  Now
$\ov{\Omega}$ acts on $V(\SG,\mu)$ ($\mu\in\mc{P}_{\ov{\mf l}}^+$)
as the scalar $(\mu+2\rho_s|\mu)_s$, and hence $Y_k^{c_\la}$
consists of those $V(\SG,\mu)$ in $Y_k$ with
$c_\la=(\mu+2\rho_s|\mu)_s$.  But such $\mu$'s are precisely of the
form $\theta(w\circ\La^\mf{x}(\la))$ by \corref{aux222}.
\end{proof}

\section{$\mf{u}_-$-homology formula of $\hgl_{\infty}$-modules at negative integral
levels}\label{sec:negative}

In this section we apply our method to obtain a Kostant-type
homology formula for irreducible highest weight
$\hgl_{\infty}$-modules that appear in bosonic Fock spaces. These
representations have negative integral levels and have been studied
in the context of vertex algebras in \cite{KR2}.

We fix a positive integer $\ell\geq 1$. Consider $\ell$ pairs of
free bosons $\gamma^{\pm,i}(z)$ with $i=1,\cdots,\ell$ (see Section
3.3). Let $\F^{-\ell}$ denote the corresponding Fock space generated
by the vaccum vector $|0\rangle$, which is annihilated by
$\gamma^{\pm,i}_r$ for $r>0$.

Given $\la\in\mc{P}({\rm GL}(\ell))$ with $\la_1\geq \cdots\geq
\la_i>0=\cdots=0>\la_{j+1}\geq \cdots\geq \la_{\ell}$, we define
\begin{equation*}
\La^{\mf
a}_-(\la)=\sum_{k=1}^i\la_k\epsilon_k+\sum_{k=j+1}^{\ell}\la_k\epsilon_{k-\ell}-\ell\La^{\mf
a}_0.
\end{equation*}
There exists a joint action of $(\hgl_{\infty},{\rm GL}(\ell))$ on
$\mf{F}^{-\ell}$, giving rise to the following multiplicity-free
decomposition:

\begin{prop} \label{duality-negative} \cite{KR2} {\rm (cf.~\cite[Theorem 5.1]{W1})} As a $(\widehat{\gl}_\infty,{\rm
GL}(\ell))$-module, we have
\begin{equation} \label{eq:dual-negative}
\F^{-\ell}\cong\bigoplus_{\la\in {\mc P}({\rm GL}(\ell))}
L(\hgl_\infty,\La^{\mf a}_-(\la))\otimes V_{{\rm GL}(\ell)}^\la.
\end{equation}
\end{prop}

Computing the trace of the operator
$\prod_{n\in\Z}x_n^{E_{nn}}\prod_{i=1}^\ell z_i^{e_{ii}}$ on both
sides of \eqnref{eq:dual-negative} (see~\cite{W1} for the action of
$(\widehat{\gl}_\infty,{\rm GL}(\ell))$ on $\F^{-\ell}$), we obtain
the following identity:
\begin{equation}\label{eq:dual-negative1}
\frac{1}{\prod_{i=1}^\ell\prod_{n\in\N}(1-x_nz_i)(1-x_{1-n}^{-1}z^{-1}_{i})}=\sum_{\la\in
{\mc P}({\rm GL}(\ell))}{\rm ch}L(\hgl_\infty,\Lambda_-^{\mf
a}(\la)){\rm ch}V^\la_{{\rm GL}(\ell)}.
\end{equation}

\begin{thm}\label{char:negative}
For $\la\in\mc{P}({\rm GL}(\ell))$, we have
\begin{equation*}
{\rm ch}L(\hgl_{\infty},\La^{\mf a}_-(\la))
=\frac{\sum_{k=0}^{\infty}\sum_{w\in
W^0_k}(-1)^ks_{(\la^+_w)'}(x_1,x_2,\cdots)
s_{(\la^-_w)'}(x^{-1}_0,x^{-1}_{-1},\cdots)}{\prod_{i,j\in\N}(1-x_{-i+1}^{-1}x_j)}.
\end{equation*}
\end{thm}

\begin{proof}
The proof is similar to that of \thmref{char:super-A}. Using
\eqnref{homology:schur} and \eqnref{char:schur} we may view
\eqnref{combid-classical1} as a symmetric function identity in the
variables $\{x_{i}^{-1}\}_{i\in-\Z_+}$ and $\{x_{i}\}_{i\in \N}$. We
then apply to this new identity the involution of the ring of
symmetric functions that interchanges the elementary symmetric
functions with the complete symmetric functions. Comparing with
\eqnref{eq:dual-negative1}, we obtain the required formula.
\end{proof}

Let $\mf{u}_{\pm}$ and $\mf{l}$ be as in
\secref{classical:homology}. For $\la\in\mc{P}({\rm GL}(\ell))$,
define the homology groups ${\rm H}_k({\mathfrak
u}_-;L(\hgl_{\infty},\La^{\mf a}_-(\la)))$ in the same  way.  One
can check that ${\rm H}_k({\mathfrak u}_-;L(\hgl_{\infty},\La^{\mf
a}_-(\la)))$ is completely reducible as an $\mathfrak l$-module, and
if $L({\mathfrak l},\eta)$ is a irreducible component, then
$(\eta+2\rho_c\vert\eta)_c=(\La^{\mf a}_-(\la)+2\rho_c\vert\La^{\mf
a}_-(\la))_c$ (cf.~\cite{J,L}).

Suppose that
$\mu=\sum_{i\in\N}\mu_i\epsilon_i-\sum_{j\in-\Z_+}\mu_j\epsilon_j+c\La^{\mf
a}_0\in\mc{P}_{\mf l}^+$ is given. Put
$(\mu_1,\mu_2,\ldots)'=(\mu_1',\mu_2',\ldots)$ and
$(\mu_0,\mu_{-1},\ldots)'=(\mu_0',\mu_{-1}',\ldots)$, and define
\begin{equation}\label{vartheta}
\vartheta(\mu)=\sum_{i\in\N}\mu'_i\epsilon_i-\sum_{j\in-\Z_+}\mu'_j\epsilon_j-c\La^{\mf
a}_0.
\end{equation}
Note that $\vartheta(\mu)\in \mc{P}_{\mf l}^+$ and $\vartheta^2$ is
the identity on $\mc{P}_{\mf l}^+$. In particular, we have
$\vartheta(\La^{\mf a}(\la))=\La^{\mf a}_-(\la)$.

Let $\omega$ be the map which has been used in the proof of
\thmref{char:negative}. That is, $\omega ({\rm ch}L(\hgl,\La^{\mf
a}(\la)))={\rm ch}L(\hgl,\La^{\mf a}_-(\la))$, $\omega ({\rm
ch}L(\mf{l},\mu))={\rm ch} L(\mf{l},\vartheta(\mu))$ for $\mu\in
\mc{P}_{\mf l}^+$, and
\begin{equation*}
\omega\big{(}{\rm ch}\big{[}\Lambda^k\mathfrak u_-\otimes
{L}(\hgl_{\infty},\La^{\mf a}(\la))\big{]}\big{)}={\rm
ch}\big{[}\Lambda^k{{\mathfrak u}}_-\otimes
{L}(\hgl_{\infty},\La^{\mf a}_-(\la))\big{]}.
\end{equation*}

Using similar arguments as in \lemref{aux112}, we can check that for
$\mu\in \mc{P}^{+}_{\mf l}$, $L(\mf{l},\mu)$ is a component in
$\Lambda^k\mathfrak u_-\otimes L(\hgl_\infty,\La^{\mf a}(\la))$ if
and only if $L(\mf{l},\vartheta(\mu))$ is a component in
$\Lambda^k\mathfrak u_-\otimes L(\hgl_\infty,\La_-^{\mf a}(\la))$.
Furthermore, they have the same multiplicity.

To compare the eigenvalues of the Casimir operator $\Omega$ on the
irreducible components in homology groups, we need the following
lemma.
\begin{lem}\cite[Lemma 4.7]{CZ1}\label{conjugatevalue}
For $\lambda\in\mc{P}^+$, we have
\begin{align*}
&(\la+2\rho_1|\la)_1 = -(\la'+2\rho_1|\la')_1-2|\la|, \\
&(\la+2\rho_2|\la)_2 = -(\la'+2\rho_2|\la')_2+2|\la|.
\end{align*}
\end{lem}

\begin{lem}\label{casimir-negative} Let $\la\in \mc{P}({\rm GL}(\ell))$ and $\mu\in \mc{P}^{+}_{\mf l}$
such that $L(\mf{l},\mu)$ is a component in $\Lambda^k\mathfrak
u_-\otimes L(\hgl_\infty,\La^{\mf a}(\la))$. Then $(\La^{\mf
a}(\la)+2\rho_c|\La^{\mf a}(\la))_c=(\mu+2\rho_c|\mu)_c$ if and only
if $(\La^{\mf a}_-(\la)+2\rho_c|\La^{\mf
a}_-(\la))_c=(\vartheta(\mu)+2\rho_c|\vartheta(\mu))_c$.
\end{lem}
\begin{proof}
Suppose that
$\mu=\sum_{i\in\N}\mu_i\epsilon_i-\sum_{j\in-\Z_+}\mu_j\epsilon_j+\ell\La^{\mf
a}_0$. Put $\mu^+=(\mu_1,\mu_2,\ldots)$ and
$\mu^-=(\mu_0,\mu_{-1},\ldots)$. Then {\allowdisplaybreaks
\begin{align*}
&(\mu+2\rho_c\vert\mu)_c \\
&=\sum_{i\in\N} \mu_i(\mu_i-2i)+\sum_{i\in\N}
\mu_{1-i}(\mu_{1-i}-2i) -\ell(\sum_{i\in\N}\mu_i+\sum_{i\in-\Z_+}\mu_i)\\
&=(\mu^++2\rho_1\vert\mu^+)_{1}+
(\mu^-+2\rho_2\vert\mu^-)_{2}-\ell(\sum_{i\in\N}\mu_i+\sum_{i\in-\Z_+}\mu_i).
\end{align*}}
On the other hand, put $(\mu^+)'=(\mu'_1,\mu'_2,\ldots)$ and
$(\mu^-)'=(\mu'_0,\mu'_{-1},\ldots)$. Then
\begin{align*}
&(\vartheta(\mu)+2\rho_c|\vartheta(\mu))_c\\
&=\sum_{i\in\N} \mu'_i(\mu'_i-2i)+\sum_{i\in\N}
\mu'_{1-i}(\mu'_{1-i}-2i) +\ell(\sum_{i\in\N}\mu'_i+\sum_{i\in-\Z_+}\mu'_i)\\
&=((\mu^+)'+2\rho_1|(\mu^+)')_1+((\mu^-)'+2\rho_2|(\mu^-)')_2+\ell(\sum_{i\in\N}\mu'_i+\sum_{i\in-\Z_+}\mu'_i)
\\
&=-(\mu^++2\rho_1|\mu^+)_1-2\sum_{i\in\N}\mu_i-(\mu^-+2\rho_2|\mu^-)_2+2\sum_{i\in-\Z_+}\mu_i+\ell(\sum_{i\in\N}\mu_i+\sum_{i\in-\Z_+}\mu_i)
\\
&=-(\mu+2\rho_c|\mu)_c-2(\mu-\ell\La^{\mf
a}_0|\sum_{i\in\Z}\epsilon_i)_c.
\end{align*}

Then from the above equations, we observe that{\allowdisplaybreaks
\begin{align*}
&(\La^{\mf a}(\la)+2\rho_c|\La^{\mf a}(\la))_c-(\mu+2\rho_c|\mu)_c \\
&=-(\La^{\mf a}_-(\la)+2\rho_c|\La^{\mf
a}_-(\la))_c+(\vartheta(\mu)+2\rho_c|\vartheta(\mu))_c-2(\La^{\mf
a}(\la)-\mu|\sum_{i\in\Z}\epsilon_i)_c.
\end{align*}}
Since $\La^{\mf a}(\la)-\mu$ is a sum of positive roots, we have
$(\La^{\mf a}(\la)-\mu|\sum_{i\in\Z}\epsilon_i)_c=0$. This completes
the proof.
\end{proof}

Using the same arguments as in \thmref{mainthm} together with Lemma
\ref{casimir-negative}, we conclude the following.
\begin{thm} Let $\la\in \mc{P}({\rm GL}(\ell))$, $k\in\Z_+$, and $\vartheta$ as in {\rm \eqnref{vartheta}}.
Then as $\mf{l}$-modules, we have
\begin{equation*}
{\rm H}_k({\mathfrak u}_-;L(\hgl_{\infty},\La^{\mf a}_-(\la))\cong
\bigoplus_{w\in W^0_k}L({\mf l},\vartheta(w\circ \La^{\mf a}(\la))).
\end{equation*}
In particular, we have ${\rm ch}\big{[}{\rm H}_k({\mathfrak
u}_-;L(\hgl_{\infty},\La^{\mf a}_-(\la))\big{]}=\omega\big{(}{\rm
ch}\big{[}{\rm H}_k({\mathfrak u}_-;L(\hgl_{\infty},\La^{\mf
a}(\la))\big{]}\big{)}$.
\end{thm}
For $\nu\in \mc{P}^+_{\mf{l}}$, let
$V(\nu)=U(\hgl_{\infty})\otimes_{U({\mf p})}L({\mf l},\nu)$ be the
generalized Verma module, where $\mf{p}=\mf{l}+\mf{u}_+$ and ${\mf
u}_+$ acts trivially on $L({\mf l},\nu)$.
\begin{cor}\label{resolution-negative} For $\la\in\mc{P}({\rm
GL}(\ell))$, there exists a resolution of $\hgl_\infty$-modules
\begin{equation*}
\cdots\stackrel{d_{k+1}}{\longrightarrow}Z_k\stackrel{d_k}
{\longrightarrow}Z_{k-1}\stackrel{d_{k-1}}{\longrightarrow}\cdots
\stackrel{d_1}{\longrightarrow}Z_0\stackrel{d_0}
{\longrightarrow}{L}(\hgl_{\infty},\La^{\mf a}_-(\la))\rightarrow 0
\end{equation*}
such that $Z_k$ has a flag of generalized Verma modules, and as an
$\mf{l}+\mf{u}_-$-module it is isomorphic to $\bigoplus_{w\in
W^0_k}V(\vartheta(w\circ\La^{\mf a}(\la)))$.
\end{cor}

\section{Index of notation in alphabetical
order}\label{final:notation}

\begin{itemize}

\item[$\cdot$] $\F^\ell$: Fock space generated by $\ell$ pairs of
free fermions defined in \secref{classical:dualpairs},

\item[$\cdot$] $\F^{-\ell}$: Fock space generated by $\ell$ pairs of
free bosons defined in \secref{sec:negative},

\item[$\cdot$] $\SF^\ell$, $\SF_0^\ell$: Fock space generated by $\ell$ pairs of
free fermions and $\ell$ pairs of free bosons defined in
\secref{super:dualpairs},

\item[$\cdot$] $\SF^\frac{1}{2}$ : Fock space generated by a pair of
neutral free fermion and free boson defined in
\secref{super:dualpairs},

\item[$\cdot$] $\G$ : the Lie algebra $\widehat{\gl}_\infty$, $\mf{b}_\infty$, $\mathfrak{c}_{\infty}$ or
$\mathfrak{d}_{\infty}$ of type $\mf{x}\in\{\mf{a,b,c,d}\}$,
respectively,

\item[$\cdot$] $\SG$ : the Lie superalgebra $\widehat{\gl}_\dinfty
$, $\widehat{\mc{B}}$, $\widehat{\mc{C}}$ or $\widehat{\mc{D}}$ of
type $\mf{x}\in\{\mf{a,b,c,d}\}$, respectively,

\item[$\cdot$] $G$ : the Lie group ${\rm GL}(\ell)$, ${\rm Pin}(2\ell)$, ${\rm
Sp}(2\ell)$ or ${\rm O}(m)$ defined in \secref{classical:dualpairs},

\item[$\cdot$] $\mathfrak{h}$ : a Cartan subalgebra of $\mathfrak{g}$,

\item[$\cdot$] $\overline{\mathfrak{h}}$ : a Cartan subalgebra of $\SG$,

\item[$\cdot$] $HS_{\la}$ : Hook Schur
function associated with $\la\in\mc{P}^+_{m|n}$, see
\secref{standardborel},

\item[$\cdot$] $I$ : an index set for simple roots for $\G$ with $0\in I$,

\item[$\cdot$] $\overline{I}$ : an index set for simple roots for $\SG$ with $0\in\ov{I}$,

\item[$\cdot$] $\mathfrak l$ : the subalgebra of $\G$ with
simple roots indexed by $I\backslash \{0\}$ \eqnref{parabolic},

\item[$\cdot$] $\overline{\mathfrak l}$ : the subalgebra of $\SG$ with
simple roots indexed by $\overline{I}\backslash \{0\}$
\eqnref{superparabolic},

\item[$\cdot$] $L(\G,\La)$: the irreducible highest weight $\G$-module of
highest weight $\La\in\h^*$,

\item[$\cdot$] $L(\SG,\La)$: the irreducible highest weight $\SG$-module of
highest weight $\La\in\overline{\h}^*$,

\item[$\cdot$] $L(\mathfrak l,\mu)$ : the irreducible $\mathfrak l$-module
of highest weight $\mu\in\h^*$,

\item[$\cdot$] $L(\ov{\mathfrak
l},\mu)$ : the irreducible $\ov{\mathfrak l}$-module of highest
weight $\mu\in\ov{\h}^*$,

\item[$\cdot$] $\mathrm{L}(\G,\La^\mf{x}(\la))$, $\mathrm{L}(\SG,\SLa^\mf{x}(\la))$ : defined in (\ref{def:L}),

\item[$\cdot$]
$\mathrm{L}(\mf{l},w\circ\La^\mf{x}(\la))$,
$\mathrm{L}(\ov{\mf{l}},\theta(w\circ\La^\mf{x}(\la)))$ : defined in
(\ref{def:LL}),

\item[$\cdot$] $\La^{\mf x}(\lambda)$ : the highest weight
of $\G$-modules in Propositions \ref{duality}-\ref{duality-b},
$\lambda \in\mathcal{P}(G)$,

\item[$\cdot$] $\SLa^{\mf x}(\lambda)$ : the highest weight
of $\SG$-modules in Propositions \ref{sduality-a}-\ref{sduality for
d}, $\lambda \in\mathcal{P}(G)$,

\item[$\cdot$] $\la_w$, $\la_w^+$, $\la_w^-$: weights in $\h^*$
associated with $\la\in\mc{P}(G)$ and $w\in W$ defined in
\secref{classical:homology},

\item[$\cdot$] $\mc{O}^{+}_{m|n}$, $\mc{O}^{++}_{m|n}$ : categories of
$\gl(m|n)$-modules defined in \secref{standardborel},

\item[$\cdot$] $\mc{O}^{++}_{\infty|\infty}$ : category of
$\gl(\infty|\infty)$-modules defined in \secref{nonstandardborel},

\item[$\cdot$] $\Omega$ : the Casimir operator of $\G$, see \secref{biformc},

\item[$\cdot$] $\ov{\Omega}$ : the Casimir operator of $\SG$, see \secref{rhos:aux},

\item[$\cdot$] $\omega^\mf{x}$ : a composition of involutions on the ring of symmetric functions such that ${\rm
ch}\mathrm{L}(\SG,\SLa^\mf{x}(\la))=\omega^\mf{x}\left( {\rm
ch}\mathrm{L}(\G,\La^\mf{x}(\la))\right)$, see
\secref{omegaappears},

\item[$\cdot$] $\mathcal{P}(G)$ : the set of highest weights for irreducible
$G$-modules in Propositions \ref{duality}-\ref{duality-b},

\item[$\cdot$] $\mc{P}^+$ : the set of partitions,

\item[$\cdot$] $\mc{P}^+_{\mathfrak l}$ : the set of
$\mf{l}$-dominant weights,

\item[$\cdot$] $\mc{P}^+_{\ov {\mathfrak l}}$ : a set of
weights given in Section 4.3,

\item[$\cdot$] $\mc{P}^+_{m|n}$ : $(m|n)$-hook partitions
defined in \secref{standardborel},

\item[$\cdot$] $\ov{\mf{p}}:=\ov{\mf l}\oplus\ov{\mf u}_+$ \eqnref{superparabolic},

\item[$\cdot$] $\rho_c$ : ``half sum" of the positive roots of $\G$, see Sections \ref{rhoc:aux1} and \ref{rhoc:aux2},

\item[$\cdot$] $\rho_s$ : ``half sum" of the positive roots of $\SG$, see \secref{rhos:aux},

\item[$\cdot$] $\rho_1$, $\rho_2$ : defined in
\secref{comp:eigenvalues},

\item[$\cdot$] $s_\la$: Schur function
associated with $\la\in\mc{P}^+$,

\item[$\cdot$] $\theta$ : a bijection from $\mc{P}_{\mf l}^+$ to
$\mc{P}_{\ov{\mf{l}}}^+$ with
$\theta(\La^\mf{x}(\la))=\SLa^\mf{x}(\la)$, see (\ref{thetamap}),

\item[$\cdot$] $\theta_1$, $\theta_2$ : defined in
\secref{comp:eigenvalues},

\item[$\cdot$] $\U_\pm$ : the nilradicals of $\G$ with $\G =\U_+ \oplus
\mathfrak l \oplus \U_-$ \eqnref{parabolic},

\item[$\cdot$] $\overline{\U}_\pm$ : the nilradicals of $\SG$ with $\SG =\overline{\U}_+ \oplus
\overline{\mathfrak l} \oplus \overline{\U}_-$
\eqnref{superparabolic},

\item[$\cdot$] $V(\nu):=U(\SG)\otimes_{U(\ov{\mf p})}L(\ov{\mf l},\nu)$ for $\nu\in
\mc{P}^+_{\ov{\mf{l}}}$,

\item[$\cdot$] $W$
(resp. $W_0$) : Weyl group of $\G$ (resp. $\mathfrak l$),

\item[$\cdot$] $W^0_k$ : the set of the
minimal length representatives of $W_0\backslash W$ of length $k$,

\item[$\cdot$] $w\circ\mu:=w(\mu+\rho_c)-\rho_c$ for $w\in W$ and $\mu\in\h^*$,

\item[$\cdot$] $(Y_k,d_k)$ : a standard resolution of
$\mathrm{L}(\SG,\SLa^\mf{x}(\la))$, see \eqnref{standardres},

\item[$\cdot$] $(Y_k^{c_{\lambda}},d_k)$ : restriction of $(Y_k,d_k)$
to generalized $c_{\lambda}$-eigenspace of  $\overline{\Omega}$, see
\eqnref{resolution:verma},

\item[$\cdot$] $\zeta$: function on half-integers defined in
\secref{biformc},

\item[$\cdot$] $(\cdot|\cdot)_c$ : a symmetric bilinear form on
$\h^*$, see \secref{biformc},

\item[$\cdot$] $(\cdot|\cdot)_s$ : a symmetric bilinear form on
$\ov{\h}^*$, see \secref{rhos:aux},

\item[$\cdot$] $(\cdot|\cdot)_1$,
$(\cdot|\cdot)_2$ : defined in \secref{comp:eigenvalues}.

\item[$\cdot$] $\langle \cdot\rangle$ : function on integers defined in
\secref{nonstandardborel}.

\end{itemize}

\end{document}